\documentclass[12pt,reqno]{amsart}
\usepackage{graphicx}
\usepackage{amssymb,amsmath,mathtools}
\usepackage{amsthm}
\usepackage{indentfirst}
\usepackage{tikz}
\usepackage{pgfplots}
\usepackage{color,graphicx}
\usepackage{multirow}
\usepackage{hyperref}
\usepackage{color}
\usepackage{epstopdf}
\usepackage{lpic}
\usepackage{enumerate}
\usepackage{caption}
\usepackage{tgpagella}
\usepackage{float}
\usepackage[super]{natbib}

\usepackage{mathrsfs}

\allowdisplaybreaks[4]

\newcommand{\be}{\begin{equation}}

\newcommand{\ee}{\end{equation}}
\newtheorem{thm}{Theorem}[section]
\newtheorem{lem}{Lemma}[section]
\newtheorem{prop}{Proposition}[section]
\newtheorem{defn}{Definition}[section]
\newtheorem{rmk}{Remark}[section]
\newtheorem{cor}{Corollary}[section]

\newtheorem{conj}{Conjecture}[section]

\newtheorem{exm}{Example}[section]

\newcommand{\R}{{\mathbb R}}

\newcommand{\A}{{\mathcal{A}}}

\numberwithin{equation}{section}
\numberwithin{figure}{section}

\pgfplotsset{compat=1.18}
\begin{document}

\title[Busemann-Selberg Functions and Completeness]{Busemann-Selberg Functions and Completeness for Dirichlet-Selberg domains in $SL(n,\mathbb{R})/SO(n,\mathbb{R})$}

\author{Yukun Du}
\address[Y. Du]{Department of Mathematics, University of Georgia, Athens, Georgia 30603}
\email{yukun.du@uga.edu}

%\subjclass[2010]{Primary 14H81, 58A50}


\begin{abstract}
We establish a general completeness criterion for Dirichlet-Selberg domains in the symmetric space \(SL(n,\mathbb{R})/SO(n)\). By introducing and analyzing \emph{Busemann-Selberg functions} - which extend classical Busemann functions and capture asymptotic behavior toward the Satake boundary - we show that every gluing manifold or orbifold produced by Dirichlet-Selberg domain is complete. This result parallels the well-known hyperbolic case and ensures that the key completeness condition in Poincar\'e's Algorithm always holds in specific cases.

\end{abstract}

\date{\today}
\maketitle
%--------------------------------------------
%|-------SECTION--1-------------------------|
%--------------------------------------------
\tableofcontents

\section{Introduction}
This paper is motivated by a semi-decidable algorithm based on Poincar\'e's Fundamental Polyhedron Theorem. The original version of Poincar\'e's Algorithm addresses the geometric finiteness of a given subgroup of $SO^+(n,1)$. It was originally proposed by Riley \cite{riley1983applications} for the case $n=3$ and was later generalized to higher dimensions by Epstein and Petronio \cite{epstein1994exposition}.
\subsection{Poincar\'e's Algorithm}
The algorithm proceeds by employing a generalization of the Dirichlet domain in hyperbolic $n$-space, as introduced in \cite{kapovich2023geometric}:
\begin{defn}
    For a point $x$ in hyperbolic $n$-space $\mathbf{H}^n$ and a discrete subset $\Gamma_0$ of the Lie group $SO^+(n,1)$, the \textbf{Dirichlet Domain} for $\Gamma_0$ centered at $x$ is defined as
    \[
    D(x,\Gamma_0) = \{y\in \mathbf{H}^n|d(g.x,y)\geq d(x,y),\ \forall g\in\Gamma_0\},
    \]
    where $g.x\in\mathbf{H}^n$ denotes the action of $g\in SO^+(n,1)$ to $x\in\mathbf{H}^n$ as an orientation-preserving isometry.
\end{defn}
This definition extends the concept of Dirichlet Domains from discrete subgroups to discrete subsets. Using this construction, Poincar\'e's algorithm can be outlined as follows:

\textbf{Poincar\'e's Algorithm for $SO^+(n,1)$.}
\begin{enumerate}
    \item Assume that a subgroup $\Gamma<SO^+(n,1)$ is given by generators $g_1,\dots,g_m$, with relators initially unknown. We begin by selecting a point $x\in\mathbf{H}^n$, setting $l = 1$, and computing the finite subset $\Gamma_l\subset \Gamma$, which consists of elements represented by words of length $\leq l$ in the letters $g_i$ and $g_i^{-1}$.
    \item Compute the face poset of the Dirichlet domain $D(x,\Gamma_l)$, which forms a finitely-sided polytope in $\mathbf{H}^n$.
    \item Utilizing this face poset data, check if $D(x,\Gamma_l)$ satisfies the following conditions:
    \begin{enumerate}
        \item Verify that $D(x,\Gamma_l)$ is an \textbf{exact convex polytope}. For each $w\in \Gamma_l$, confirm that the isometry $w$ pairs the two facets contained in $\mathrm{Bis}(x,w.x)$ and $\mathrm{Bis}(x,w^{-1}.x)$, provided these facets exist.
        \item Verify that $D(x,\Gamma_l)$ satisfies the \textbf{tiling condition}, meaning that the quotient space $M$ obtained by identifying the paired facets of $D(x,\Gamma_l)$ is an $\mathbf{H}^n$-orbifold. This condition is formulated as a \textbf{ridge-cycle condition}, as described in \cite{ratcliffe1994foundations}.
        \item Verify that each generator $g_i$ can be expressed as a product of the facet pairings of $D(x,\Gamma_l)$, following the procedure in \cite{riley1983applications}.
    \end{enumerate}
    \item If any of these conditions are not met, increment $l$ by $1$ and repeat the initialization, computation and verification processes.
    \item If all conditions are satisfied, the quotient space of $D(x,\Gamma_l)$ is complete \cite{kapovich2023geometric}. By Poincar\'e's Fundamental Polyhedron Theorem, $D(x,\Gamma_l)$ is a fundamental domain for $\Gamma$, and $\Gamma$ is geometrically finite. Specifically, $\Gamma$ is discrete and has a finite presentation derived from the ridge cycles of $D(x,\Gamma_l)$ \cite{ratcliffe1994foundations}.
\end{enumerate}

The completeness condition is fundamental when applying Poincar\'e's Fundamental Polyhedron Theorem. If the quotient of the convex polytope $D$ by its facet pairing is incomplete, the facet-pairing transformations may generate additional relators. In the context of hyperbolic $3$-space, this phenomenon is closely related to \textbf{Hyperbolic Dehn fillings}\cite{thurston1982three}.

Consider, for instance, the \textbf{Meyerhoff manifold}\cite{meyerhoff1987lower}, which arises from completing an incomplete gluing of a certain ideal triangular bipyramid. This manifold corresponds to the $(5,1)$-Dehn filling on the \textbf{figure-eight knot complement}. While the ridge cycles of the Meyerhoff manifold provide relators for the facet pairings that agree combinatorially with those of the figure-eight knot group, additional relators emerge due to the Dehn filling condition,\cite{purcell2020hyperbolic}. Consequently, Poincar\'e's Fundamental Polyhedron Theorem cannot fully recover the group presentation generated by Meyerhoff facet-pairing transformations mentioned above.

Fortunately, for Dirichlet domains, the completeness condition is not a concern - as noted in Step (5) of Poincar\'e's Algorithm. The guaranteed satisfaction of the completeness condition can be explained through the concept of \textbf{Busemann Functions}, \cite{busemann1955geometry}:
\begin{defn}\label{def:1:2}
    Let $a\in \partial \mathbf{H}^n$ be an ideal point and $x\in \mathcal{H}^n$ be a reference point. For any geodesic ray $\gamma: \mathbb{R}\to \mathbf{H}^n$ asymptotic to $a$, and for any $y\in \mathbf{H}^n$, the limit
    \[
    b_{a,x}(y): = \lim_{t\to\infty}d(\gamma(t),y) - d(\gamma(t),x)
    \]
    exists and is independent of the choice of $\gamma$. This limit defines the Busemann function $b_{a,x}: \mathbf{H}^n\to \mathbb{R}$.
\end{defn}
It is well-known that the Busemann function satisfies the following asymptotic behavior:
\begin{itemize}
    \item If $\gamma$ is a geodesic ray asymptotic to $a$, then $\lim_{t\to\infty}b_{a,x}(\gamma(t)) = 0$.
    \item If $\gamma$ is any geodesic ray asymptotic to a different ideal point, then $\lim_{t\to\infty}b_{a,x}(\gamma(t)) = \infty$.
\end{itemize}
One considers the level sets of the Busemann functions, known as \textbf{horospheres} in $\mathbf{H}^n$. In the Poincar\'e disk model, horospheres are represented as $(n-1)$-spheres tangent to the visual boundary at the base points. For a finite-volume convex polytope, horospheres based at its ideal vertices serve to separate the cusp parts from the remainder of the polytope.

For Dirichlet Domains, the Busemann function exhibits the following invariance property:
\begin{lem}[\cite{kapovich2023geometric}]\label{lem:1:1}
    Let $D = D(x,\Gamma_0)$ be the Dirichlet Domain for a finite subset $\Gamma_0\subset SO^+(n,1)$ with center $x\in \mathbf{H}^n$, satisfying the following conditions:
    \begin{itemize}
        \item $D$ is exact: For each $g\in \Gamma_0$, we have $g^{-1}\in \Gamma_0$, and the two facets of $D$ contained in $\mathrm{Bis}(x,g.x)$ and $\mathrm{Bis}(x,g^{-1}.x)$ are isometric under the action of $g$.
        \item $D$ is finite-volume, i.e., $\overline{D}\cap \partial \mathbf{H}^n$ is a discrete set of ideal points.
    \end{itemize}
    Let $a\in \partial \mathbf{H}^n\cap \overline{D}$ be an ideal vertex, and suppose $g_1,\dots, g_m\in \Gamma_0$. Define the sequence of ideal points inductively as follows: $a_0 = a$ and $a_i = g_i.a_{i-1}$ for $i=1,\dots, m$. If the following conditions are satisfied:
    \begin{itemize}
        \item $\mathrm{Bis}(x,g_i.x)$ contains a certain facet of $D$ for $i=1,\dots,m$.
        \item The points $a_i$, $i=0,\dots, m$ are ideal vertices of $D$.
        \item The sequence satisfies $a_m = a_0$.
    \end{itemize}
    Then the word $w = g_m\dots g_1$ preserves the Busemann function based at $a$, i.e.,
    \[
    b_{a,x}(y) = b_{a,x}(w.y),\ \forall y\in \mathbf{H}^n.
    \]
\end{lem}
This invariance ensures that Cauchy sequences in the cusp region of the quotient $D/\sim$ remain bounded away from the visual boundary, thereby guaranteeing the completeness condition in Step (5) of Poincar\'e's Algorithm:
\begin{thm}[\cite{kapovich2023geometric}]\label{thm:1:1}
    Let $D = D(X,\Gamma_l)$ be a finitely-sided Dirichlet domain in $\mathbf{H}^n$ satisfying the tiling condition. Then the quotient space $M = D/\!\sim$ is complete. In particular, $D$ is a fundamental domain for the subgroup generated by its facet pairings.
\end{thm}
This property of the Dirichlet domain simplifies the implementation of Poincar\'e's Algorithm for $SO^+(n,1)$.

\subsection{The Symmetric Space \texorpdfstring{$SL(n,\mathbb{R})/SO(n)$}{Lg}}
Our research seeks to generalize Poincar\'e's Algorithm, extending it to other Lie groups, particularly $SL(n,\mathbb{R})$. It is well-established that $SL(n,\mathbb{R})$ acts as the orientation-preserving isometry group on the symmetric space $SL(n,\mathbb{R})/SO(n)$, \cite{eberlein1996geometry}. We recognize this space through the following models:
\begin{defn}
    The \textbf{hypersurface model} of $SL(n,\mathbb{R})/SO(n)$ is defined as the set
    \begin{equation}
        \mathcal{X}_n = \mathcal{X}_{n,\mathrm{hyp}} = \{X\in \mathrm{Sym}_n(\mathbb{R})\,|\,\det(X) = 1,\ X>0\},
    \end{equation}
    equipped with the metric tensor
    \[
    \langle A,B\rangle_X = \mathrm{tr}(X^{-1}AX^{-1} B),\ \forall A,B\in T_X\mathcal{X}_n.
    \]
    Here, $\mathrm{Sym}_n(\mathbb{R})$ denotes the vector space of $n\times n$ real symmetric matrices, and $X>0$ (or $X\geq 0$) indicates that $X$ is positive definite (or positive semi-definite, respectively). Throughout the paper, we adopt the bilinear form $\langle A,B\rangle: = \mathrm{tr}(A\cdot B)$ on $\mathrm{Sym}_n(\mathbb{R})$ and interpret orthogonality accordingly.
\end{defn}
In this model, the action of $SL(n,\mathbb{R})$ on $\mathcal{X}_n$ is given by
\[
SL(n,\mathbb{R})\curvearrowright \mathcal{X}_n,\ g.X = g^{\mathsf{T}} Xg.
\]
An alternative model is also considered in the paper:
\begin{defn}
    The \textbf{projective model} of $\mathcal{X}_n$ is defined as follows:
    \begin{equation}
        \mathcal{X}_n = \mathcal{X}_{n,\mathrm{proj}} = \{[X]\in \mathbf{P}(\mathrm{Sym}_n(\mathbb{R}))\,|\,X>0\}.
    \end{equation}
\end{defn}
It is evident that the two models of the symmetric space $\mathcal{X}_n$ are diffeomorphic.

Classic Dirichlet domains in $\mathcal{X}_n$ are non-convex and often impractical for further study. To overcome these challenges, our generalization of Poincar\'e's Algorithm utilizes an $SL(n,\mathbb{R})$-invariant proposed by Selberg\cite{selberg1962discontinuous} as a substitute for the Riemannian distance on $\mathcal{X}_n$.
\begin{defn}
    For $X,Y\in \mathcal{X}_n$, the \textbf{Selberg invariant} from $X$ to $Y$ is defined as
    \[
    \mathfrak{s}(X,Y) = \mathrm{tr}(X^{-1}Y).
    \]

    For a point $X\in \mathcal{X}_n$ and a discrete subset $\Gamma_0\subset SL(n,\mathbb{R})$, the \textbf{Dirichlet-Selberg Domain} for $\Gamma$ centered at $X$ is defined as
    \[
    DS(X,\Gamma_0) = \{Y\in\mathcal{X}_n|\mathfrak{s}(g.X,Y)\geq \mathfrak{s}(X,Y),\ \forall g\in\Gamma_0\}.
    \]
\end{defn}
Dirichlet-Selberg domains serve as fundamental domains when $\Gamma<SL(n,\mathbb{R})$ is a discrete subgroup satisfying $Stab_\Gamma(X) = \mathbf{1}$, \cite{kapovich2023geometric}. Moreover, these domains are realized as convex polyhedra in $\mathcal{X}_n$, defined as follows:
\begin{defn}
    A $k$-dimensional \textbf{plane} of $\mathcal{X}_n$ is the non-empty intersection of a $(k+1)$-dimensional linear subspace of $\mathrm{Sym}_n(\mathbb{R})$ with $\mathcal{X}_{n,\mathrm{hyp}}$. An $(n-1)(n+2)/2 - 1$-dimensional plane is referred to as a \textbf{hyperplane} of $\mathcal{X}_n$.

    \textbf{Half spaces} and \textbf{convex polyhedra} in $\mathcal{X}_n$ are defined analogously to the corresponding concepts in hyperbolic spaces \cite{ratcliffe1994foundations}.

    For a convex polytope $D$ in $\mathcal{X}_n$, its \textbf{faces}, \textbf{facets}, and \textbf{ridges} are also defined analogously. We denote the collections of these objects by $\mathcal{F}(D)$, $\mathcal{S}(D)$, and $\mathcal{R}(D)$, respectively.
\end{defn}
Hyperplanes in $\mathcal{X}_n$ can be realized as \textbf{perpendicular planes}. For any indefinite matrix $A\in \mathrm{Sym}_n(\mathbb{R})$, the set
\[
A^\perp = \{X\in\mathcal{X}_n|\mathrm{tr}(A.X) = 0\},
\]
is non-empty, and constitutes a hyperplane of $\mathcal{X}_n$, \cite{finsler1936vorkommen,Du2024-dp}. Specifically, the boundary of a Dirichlet-Selberg domain $DS(X,\Gamma)$ consists of bisectors:
\[
\mathrm{Bis}(X,g.X) = \{Y\in\mathcal{X}_n|\mathfrak{s}(X,Y) = \mathfrak{s}(g.X,Y)\},
\]
for $g\in \Gamma$. In the form of perpendicular planes, these bisectors are expressed as
\[
\mathrm{Bis}(X,g.X) = \left(X^{-1} - (g.X)^{-1}\right)^\perp.
\]
These facts provide suitable analogs to corresponding concepts in hyperbolic spaces for our proposed generalization of Poincar\'e's Algorithm to $SL(n,\mathbb{R})$.

In \cite{kapovich2023geometric}, a generalized version of Poincar\'e's Algorithm was proposed, adopting Dirichlet-Selberg domains in the Dirichlet construction process. Details of this algorithm are reviewed in Section \ref{Sec:2}.
\subsection{The Main Result}
The main purpose of this paper is to generalize Theorem \ref{thm:1:1} - the completeness property for hyperbolic Dirichlet domains - to Dirichlet-Selberg domains in $\mathcal{X}_n$. We focus on Dirichlet-Selberg domains of finite volume, which correspond to \textbf{lattices} in $SL(n,\mathbb{R})$. These subgroups play an important role among the discrete subgroups of $SL(n,\mathbb{R})$. In particular, the quotients of finite volume Dirichlet-Selberg domains exhibit favorable structures. By exploiting these properties and extending the approach in \cite{ratcliffe1994foundations}, we establish the following result:
\begin{thm}\label{thm:1:2}
Let $D = DS(X,\Gamma_0)$ be an exact partial Dirichlet-Selberg domain centered at $X\in \mathcal{X}_3$, defined with respect to a finite set $\Gamma_0\subset SL(3,\mathbb{R})$, and satisfying the tiling condition. If, in addition, $D$ has finite volume, then the quotient of $D$ under its intrinsic facet pairing is complete.
\end{thm}
The proof of Theorem \ref{thm:1:2} proceeds by constructing a family of generalized Busemann functions on $\mathcal{X}_3$, which possess specific invariance properties under the action of $SL(3,\mathbb{R})$. Moreover, we isolate the cusp regions of $D$ via generalized horospheres, analogous to the hyperbolic setting. Furthermore, we formulate these Busemann function constructions in the general setting of $\mathcal{X}_n$.
\subsection{Organization of the Paper}
This paper is structured as follows. In Section \ref{Sec:2}, we review the generalized Poincar\'e's Algorithm for the group $SL(n,\mathbb{R})$, and the compactifications of $\mathcal{X}_n$. In Section \ref{Sec:3}, we introduce the key construction - \textbf{Busemann-Selberg functions} on $\mathcal{X}_n$, define generalized horospheres via these functions, and study the structure of finite-volume Dirichlet-Selberg domains. In Section \ref{Sec:4}, we establish various properties of Busemann-Selberg functions, horospheres and hyperplanes, which are required for the proof of the main theorem. Section \ref{Sec:5} presents the proof of Theorem \ref{thm:1:2}, synthesizing earlier results. Finally, we give a concrete example in Section \ref{Sec:6}, constructing an exact finitely-sided Dirichlet-Selberg domain in $\mathcal{X}_n$ that illustrates our main results.

%%
%% - Move the Def of Satake compactification to here
%% - Add Satake compactification (Borel & Ji) and ideal compactification stuffs, especially page 60 to 62
% - Rename into "The symmetric space SL(n,R)" or so
\section{Preliminaries for the Symmetric Space \texorpdfstring{$\mathcal{X}_n$}{Lg}}\label{Sec:2}
\subsection{Poincar\'e's Algorithm for \texorpdfstring{$SL(n,\mathbb{R})$}{Lg}}
Let us recall the Poincar\'e's Algorithm on $\mathcal{X}_n$ described in \cite{kapovich2023geometric,Du2024-dp}, analogically to the real hyperbolic case.

We start by considering the \textbf{facet pairings} for convex polytopes in $\mathcal{X}_n$. These are analogous to the hyperbolic case:
\begin{defn}
    A convex polytope $D$ in $\mathcal{X}_n$ is said to be \textbf{exact} if, for each of its facets $F$, there exists an element $g_F\in SL(n,\mathbb{R})$ such that 
    \[
    F = D\cap g_F.D,
    \]
    and such that $F': = g_F^{-1}.F$ is also a facet of $D$. The transformation $g_F$ is referred to as a \textbf{facet pairing transformation} for the facet $F$.

    For an exact convex polytope $D$, a \textbf{facet pairing} is a set
    \[
    \Gamma_0 = \{g_F\in SL(n,\mathbb{R})|F\in \mathcal{S}(D)\},
    \]
    where each facet $F$ is assigned a facet pairing transformation $g_F$, and the transformations satisfy $g_{F'} = g_F^{-1}$ for every paired facets $F$ and $F'$.
\end{defn}
For a discrete subgroup $\Gamma<SL(n,\mathbb{R})$, the Dirichlet-Selberg domain $D = DS(X,\Gamma)$ has a canonical facet pairing. Each element $g\in\Gamma$ serves as the facet-pairing transformation between the facets contained in the bisectors $\mathrm{Bis}(X,g^{-1}.X)$ and $\mathrm{Bis}(X,g.X)$, provided these facets exist.

A facet pairing naturally defines an equivalence relation on $D$:
\begin{defn}
    Two points $X,X'$ in $D$ are said to be \textbf{paired} if $X\in F$, $X'\in F'$, and $g_F^{-1}.X = X'$ for a specific pair of facets $F$ and $F'$. This pairing defines a binary relation, denoted by $X\cong X'$. The equivalence relation generated by this binary relation is denoted by $\sim$.

    The \textbf{cycle} of a point $X$ in an exact convex polytope $D$ with a facet pairing $\Gamma_0$ is the equivalence class of $X$ under the relation induced by $\Gamma_0$.
\end{defn}
With the preliminaries above, we introduce the \textbf{tiling condition} involved in Poincar\'e's Algorithm:
\begin{defn}
    For an exact convex polytope $(D,\Gamma_0)$ in $\mathcal{X}_n$, the equivalence relation $\sim$ defines a quotient space $M = D/\sim$. The polytope is said to satisfy the \textbf{tiling condition} if the corresponding quotient space $M$, equipped with the path metric induced from $\mathcal{X}_n$, has the structure of a $\mathcal{X}_n$-manifold or orbifold.
\end{defn}
The tiling condition can be reformulated using a \textbf{ridge cycle condition}, analogous to the hyperbolic case described in \cite{ratcliffe1994foundations}. However, unlike hyperbolic polytopes, the dihedral angles between two facets of a $\mathcal{X}_n$-polytope depend on the choice of the base point. This dependency is further explored in Subsection \ref{subsec:3:2}. Nevertheless, the formulation of the ridge cycle condition remains valid when the base point is specified:
\begin{defn}
    Let $X$ be a point in the interior of a ridge $r$ of the polytope $D$. The cycle $[X]$ is said to satisfy the \textbf{ridge cycle condition} if the following criteria are met:
    \begin{itemize}
        \item The ridge cycle $[X]$ is a finite set $\{X_1,\dots,X_m\}$, and
        \item The dihedral angle sum satisfies
        \[
        \theta[X] = \sum_{i=1}^m\theta(X_i) = 2\pi/k,
        \]
        for certain $k\in\mathbb{N}$. Here, $\theta(X_i)$ denotes the Riemannian dihedral angle between the two facets containing $X_i$, measured at the point $X_i$.
    \end{itemize}
\end{defn}
In \cite{Du2024-dp}, we reformulate the ridge cycle condition by introducing a generalized angle-like function that does not depend on the choice of base points. This approach applies to generic pairs of hyperplanes, simplifying the implementation of Poincar\'e's Algorithm.

Using the framework explained above, we propose a generalized Poincar\'e's Algorithm for the Lie group $SL(n,\mathbb{R})$, parallel to the classical algorithm for $SO^+(n,1)$:\cite{kapovich2023geometric,Du2024-dp}

\textbf{Poincar\'e's Algorithm for $SL(n,\mathbb{R})$.}
\begin{enumerate}
    \item Assume that a subgroup $\Gamma<SL(n,\mathbb{R})$ is given by generators $g_1,\dots,g_m$, with relators initially unknown. We begin by selecting a point $X\in\mathcal{X}_n$, setting $l = 1$, and computing the finite subset $\Gamma_l\subset \Gamma$, which consists of elements represented by words of length $\leq l$ in the letters $g_i$ and $g_i^{-1}$.
    \item Compute the face poset of the Dirichlet-Selberg domain $DS(X,\Gamma_l)$, which forms a finitely-sided polytope in $\mathcal{X}_n$.
    \item Utilizing this face poset data, check if $DS(X,\Gamma_l)$ satisfies the following conditions:
    \begin{enumerate}
        \item Verify that $DS(X,\Gamma_l)$ is an \textbf{exact convex polytope}. For each $w\in \Gamma_l$, confirm that the isometry $w$ pairs the two facets contained in $\mathrm{Bis}(X,w.X)$ and $\mathrm{Bis}(X,w^{-1}.X)$, provided these facets exist.
        \item Verify that $D(X,\Gamma_l)$ satisfies the \textbf{tiling condition}, which is introduced above.
        \item Verify that each element $g_i$ can be expressed as a product of the facet pairings of $DS(X,\Gamma_l)$, following the procedure in \cite{riley1983applications}.
    \end{enumerate}
    \item If any of these conditions are not met, increment $l$ by $1$ and repeat the initialization, computation, and verification processes.
    \item If all conditions are satisfied, we verify if the quotient space of $DS(X,\Gamma_l)$ is complete. If so, by Poincar\'e's Fundamental Polyhedron Theorem, $DS(X,\Gamma_l)$ is a fundamental domain for $\Gamma$, and $\Gamma$ is geometrically finite. Specifically, $\Gamma$ is discrete and has a finite presentation derived from the ridge cycles of $DS(X,\Gamma_l)$.
\end{enumerate}
Several questions arise from this algorithm. As a semi-decidable procedure, it is clear that for a given center \(X\in \mathcal{X}_n\) and subgroup \(\Gamma < SL(n,\mathbb{R})\), the algorithm terminates in finite time if and only if the Dirichlet-Selberg domain \(DS(X,\Gamma)\) is finitely-sided. It remains unknown whether this condition holds for nonuniform lattices \cite{kapovich2023geometric}. Davalo and Riestenburg \cite{davalo2024finitesideddirichletdomainsanosov} showed that uniform lattices in \(SO(n-1,1)\), when regarded as subgroups of \(SL(n,\mathbb{R})\) via the canonical inclusion 
\[
SO(n-1,1)\hookrightarrow SL(n,\mathbb{R}),
\]
do not admit finitely-sided Dirichlet-Selberg domains for any center. This result gives a negative answer to Kapovich’s question on whether Anosov subgroups always admit finitely-sided Dirichlet-Selberg domains.

Davalo and Riestenburg also considered the \(\lvert \log \omega_i\rvert\)-undistorted subgroups of \(SL(2n,\mathbb{R})\), proving that these subgroups admit finitely-sided Dirichlet-Selberg domains for every choice of center. The \(\lvert \log \omega_i\rvert\)-undistorted property holds for the Schottky groups in \(SL(2n,\mathbb{R})\) we constructed in \cite{Du2024-dp}, but does not extend to subgroups of \(SL(2n-1,\mathbb{R})\).

Another question concerns the completeness condition for Dirichlet-Selberg domains in $\mathcal{X}_n$. This condition is required by Poincar\'e's Fundamental Polyhedron Theorem and holds for all hyperbolic Dirichlet domains with the tiling condition (see Theorem \ref{thm:1:1}). Kapovich conjectured an analogous property holds for Dirichlet-Selberg domains:
\begin{conj}[\cite{kapovich2023geometric}]
    Let $D = DS(X,\Gamma_l)$ be a finitely-sided Dirichlet-Selberg domain in $\mathcal{X}_n$ satisfying the tiling condition. Then the quotient space $M = D/\sim$ is complete. In particular, $D$ is a fundamental domain for the subgroup generated by its facet pairings.
\end{conj}
This conjecture motivates our main result, which establishes the same conclusion under the additional hypothesis that \(D\) has finite volume.
\subsection{Compactifications of \texorpdfstring{$\mathcal{X}_n$}{Lg}}
In this paper We employ several compactifications of the symmetric space \(\mathcal{X}_n\). In particular, the \textbf{Satake compactification} arises naturally from the polyhedral structure of Dirichlet-Selberg domains, while the \textbf{visual compactification} is essential for studying the geometry and completeness of \(\mathcal{X}_n\)-manifolds.

Satake \cite{satake1960representations} introduced a family of compactifications of symmetric spaces associated to faithful finite-dimensional representations of the ambient Lie group.  The \textbf{standard} Satake compactification of \(\mathcal{X}_n\) corresponds to the identity representation of \(SL(n,\mathbb{R})\) and admits the following description via a projective model \cite{borel2006compactifications}:
\begin{defn}
    The \textbf{standard Satake compactification} of $\mathcal{X}_n$ is
    \[
    \overline{\mathcal{X}_n}^S = \{\widetilde{X}\in \mathbf{P}(\mathrm{Sym}_n(\mathbb{R}))\mid X\geq 0\},
    \]
    endowed with the projective topology on $\mathrm{Sym}_n(\mathbb{R})$. The \textbf{Satake boundary} is
    \[
    \partial_S \mathcal{X}_n = \overline{\mathcal{X}_n}^S\backslash \mathcal{X}_n.
    \]
\end{defn}
When the context is clear we shall omit the superscript \(S\) and simply write \(\overline{\mathcal{X}}_n\).
\begin{prop}[\cite{borel2006compactifications}]
    The standard Satake compactification decomposes as
    \[
    \overline{\mathcal{X}_n} = \mathcal{X}_n\sqcup\bigsqcup_{k=1}^{n-1}\left(SL(n,\mathbb{R})\mathcal{X}_k\right),
    \]
    where each $\mathcal{X}_k = SL(k,\mathbb{R})/SO(k)$ embeds into $\partial_S\mathcal{X}_n$ via
    \[
    \mathcal{X}_k\hookrightarrow \overline{\mathcal{X}_n},\ X\mapsto \mathrm{diag}(X,O_{n-k}),
    \]
    and $SL(n,\mathbb{R})$ acts by congruence on the set of semi-definite matrices.
\end{prop}
More generally, for any $g\in SL(n,\mathbb{R})$ and $k=1,\dots,n-1$, the image $g.\mathcal{X}_k\subset \partial_S\mathcal{X}_n$ is called a \textbf{Satake boundary component}. Under the projective model, $g.\mathcal{X}_k$ identifies with the set of positive semidefinite \(n\times n\) matrices of rank \(k\) whose column space is the \(k\)-subspace
\[
\mathrm{span}(g.\mathbf{e}_1,\dots,g.\mathbf{e}_k)\subset \mathbb{R}^n.
\]
We denote by \(\partial_S(V)\) the boundary component corresponding to a linear subspace \(V\subset\mathbb{R}^n\), and we say its \textbf{type} is \(k=\dim V\). If \(V=\mathrm{span}(\mathbf v_1,\dots,\mathbf v_k)\), we write
\[
  \partial_S V \;=\;\partial_S(\mathbf v_1,\dots,\mathbf v_k).
\]
Whenever \(W\subset V\), the boundary component \(\partial_S(W)\) lies in the boundary of \(\partial_S(V)\); we express this by $\partial_S(W)<\partial_S(V)$, and write $\partial_S(W)\leq\partial_S(V)$ if $W\subseteq V$. The \textbf{compactification} of \(\partial_S(V)\) is the disjoint union of all subordinate components:
\[
\overline{\partial_S(V)} = \bigsqcup_{W\subseteq V}\partial_S(W).
\]
Dually, the \textbf{star} of \(\partial_S(V)\) consists of those components whose subspaces contain \(V\):
\[
\mathrm{st}(\partial_S(V)) = \bigsqcup_{U\supseteq V, U\subsetneq \mathbb{R}^n}\partial_S(U).
\]
Since all type-\(k\) components lie in a single $SL(n,\mathbb{R})$-orbit, each \(\partial_S(V)\) is diffeomorphic to the symmetric space \(\mathcal X_k\).  More precisely:
\begin{defn}\label{def:3:5}
    Let \(V\subset\R^n\) be a \(k\)-dimensional subspace and choose an orthonormal basis \(\iota_V=(v_1,\dots,v_k)\in\R^{n\times k}\). Define
    \[
      \pi_V\colon \partial_S(V)\;\longrightarrow\;\mathcal X_k,
      \qquad
      \pi_V(\alpha)
      =
      \frac{\iota_V^{\mathsf{T}}\,\alpha\,\iota_V}
           {\bigl(\det(\iota_V^{\mathsf{T}}\,\alpha\,\iota_V)\bigr)^{1/k}}.
    \]
    This extends to a projection \(\pi_V\colon \mathcal X_n\sqcup\mathrm{st}(\partial_S(V))\to\mathcal X_k\). In the projective model one similarly obtains
    \[
      \pi_V\colon \overline{\mathcal X_n}\;\longrightarrow\;\overline{\mathcal X_k}, \quad \pi_V(\widetilde{X}) = \widetilde{\iota_V^\mathsf{T}X\iota_V}.
    \]
    The map \(\pi_V\) is well-defined up to the action of \(SO(k)\).
\end{defn}

As with any non-compact symmetric space, \(\mathcal X_n\) admits a \textbf{visual compactification} obtained by adjoining equivalence classes of geodesic rays.
\begin{defn}[\cite{eberlein1996geometry}]
    A \textbf{geodesic ray} in \(\mathcal{X}_n\) may be written as
    \[
    \gamma(t) = g.\exp(tA),\quad g\in SL(n,\mathbb{R}),\ A\in\mathfrak{sl}(n,\mathbb{R}),\ A^\mathsf{T} = A.
    \]
    Two rays \(\gamma_1,\gamma_2\) are equivalent, \(\gamma_1\sim\gamma_2\), if
    \[
    \overline{\lim}_{t\to\infty}(\gamma_1(t),\gamma_2(t))<\infty.
    \]
    The \textbf{visual boundary} \(\partial_\infty\mathcal X_n\) is the set of equivalence classes of geodesic rays in $\mathcal{X}_n$, and the \textbf{visual compactification} is
    \[
    \overline{\mathcal{X}_n}^\infty = \mathcal{X}_n\sqcup \partial_\infty\mathcal{X}_n,
    \]
    endowed with the cone topology.
\end{defn}
The visual boundary \(\partial_\infty\mathcal X_n\) carries the structure of a spherical building, identified with the \textbf{complex of flags} in \(\R^n\).
\begin{defn}[\cite{bridson2013metric}]
  The \textbf{complex of flags} in \(\mathbb{R}^n\) is the simplicial complex whose \(k\)-simplices correspond to flags
  \[
    V_\bullet = V_1\subset V_2\subset\dots\subset V_{k+1}\subset\mathbb{R}^n,
  \]
  for $k=0,\dots, n-2$. The facets of a \(k\)-simplex are obtained by deleting one subspace from the flag.
\end{defn}
\begin{prop}[\cite{bridson2013metric}]
    There is an \(SL(n,\R)\)-equivariant bijection from $\partial_\infty\mathcal{X}_n$ to the complex of flags in $\mathbb{R}^n$. Concretely, if $A$ has distinct eigenvalues $\lambda_1>\dots>\lambda_k$ with corresponding eigenspaces $W_1,\dots, W_k$, then the ray $\gamma(t) = g.\exp(t A)$ determines the flag
    \[
    g^{-1}.W_1\subset g^{-1}.(W_1\oplus W_2)\subset\dots\subset g^{-1}.\oplus_{i=1}^{k-1}W_{i}\subset \mathbb{R}^n.
    \]
\end{prop}
In particular, the vertices of $\partial_\infty \mathcal{X}_n$ correspond to linear subspaces \(V\subset\mathbb{R}^n\). We denote the ideal vertex associated to \(V\) by \(\xi_V\) and call its \textbf{type} \(k=\dim V\). Moreover, for each subspace $V\subset \mathbb{R}^n$, the stabilizer of \(\partial_S(V)\subset\partial_S\mathcal X_n\) in \(SL(n,\R)\) coincides with
the stabilizer of \(\xi_V\), namely the maximal parabolic subgroup
preserving \(V\) \cite{borel2006compactifications}.
% Consider move these to the section:
% Satake planes -> boundary component
% The diffeo
% Also, let me avoid the using of \partial_S V^\perp
\vspace{12pt}
\section{Satake Faces, Busemann-Selberg Function, and Horoballs}\label{Sec:3}
In this section we extend the classical Busemann function to define \textbf{Busemann-Selberg functions} and their level sets (\textbf{horoballs}) in $\mathcal{X}_n$. We then examine the polyhedral structure of finite-volume Dirichlet-Selberg domains and introduce the notions of \textbf{Satake boundary components} and \textbf{Satake faces} of such domains. These constructions are essential in the proof of our main theorem.
\subsection{Busemann-Selberg Functions and Horoballs in \texorpdfstring{$\mathcal{X}_n$}{Lg}}
The classical Busemann function is defined as a limit of distance differences in hyperbolic space (see Definition~\ref{def:1:2}). We generalize this concept by replacing the hyperbolic distance by Selberg’s invariant \(\mathfrak s\) on the symmetric space $\mathcal{X}_n$.
\begin{defn}
    Let $X\in\mathcal{X}_n$ and $\alpha\in \partial_S\mathcal{X}_n$. Choose any path $A(t)\subset \mathcal{X}_n$ with
    \[
    \lim_{t\to \infty}A(t) = \alpha.
    \]
    The \textbf{(type $0$) Busemann-Selberg function} based at $\alpha$ with reference point $X$ is
    \[
    \mathfrak{b}_{\alpha,X}: \mathcal{X}_n\to \mathbb{R}_+,\quad \mathfrak{b}_{\alpha,X}(Y) = \lim_{t\to \infty}\frac{\mathfrak{s}(Y,A(t))}{\mathfrak{s}(X,A(t))}.
    \]
\end{defn}
\begin{rmk}
    If $\alpha$ is represented by a singular positive semi-definite matrix (also denoted \(\alpha\)), one obtains the closed-form
    \[
      \mathfrak{b}_{\alpha,X}(Y) = \frac{\mathrm{tr}(Y^{-1}\alpha)}{\mathrm{tr}(X^{-1}\alpha)},\quad \forall Y\in\mathcal{X}_n,
    \]
    which is independent of the choice of matrix representative for $\alpha$.
\end{rmk}
The proof of the main theorem requires the following generalization of the Busemann-Selberg function, obtained by composing a type-$0$ Busemann-Selberg function with the projection onto a Satake boundary component.
\begin{defn}\label{def:3:7}
    Let $X\in \mathcal{X}_n$ and $\Pi$ is a boundary component of type $n-k$. Suppose $\alpha$ lies on $\partial\Pi$, so that $\mathrm{rank}(\alpha)<n-k$. Let
    \[
    \pi: \overline{\mathcal{X}_n}\to \overline{\mathcal{X}_{n-k}}
    \]
    be the projection associated to \(\Pi\) (cf. Definition \ref{def:3:5}). The \textbf{type-$k$ Busemann-Selberg function} based at $(\alpha,\Pi)$ with reference point $X$ is
    \[
    \mathfrak{b}^{(k)}_{\Pi;\alpha,X}: \mathcal{X}_n\to \mathbb{R}_+,\quad \mathfrak{b}^{(k)}_{\Pi;\alpha,X}(Y) = \mathfrak{b}_{\pi(\alpha),\pi(X^{-1})^{-1}}(\pi(Y^{-1})^{-1}) = \frac{\mathrm{tr}(\pi(Y^{-1})\pi(\alpha))}{\mathrm{tr}(\pi(X^{-1})\pi(\alpha))}.
    \]
\end{defn}
In concrete terms, if \(\iota_\Pi\in\mathbb{R}^{n\times(n-k)}\) has orthonormal columns spanning \(\Pi\), one checks
\begin{equation}\label{equ:3:2}
  \mathfrak{b}^{(k)}_{\Pi;\alpha,X}(Y) = \frac{\mathrm{tr}(Y^{-1}\alpha)\det(\iota_\Pi^{\mathsf{T}}Y^{-1}\iota_\Pi)^{-1/(n-k)}}{\mathrm{tr}(X^{-1}\alpha)\det(\iota_\Pi^{\mathsf{T}}X^{-1}\iota_\Pi)^{-1/(n-k)}}.
\end{equation}
If one replaces \(\iota_\Pi\) by \(\iota_\Pi\,Q\) with \(Q\in SO(n-k)\), then
\[
    \det((\iota_\Pi Q)^\mathsf{T}X^{-1}(\iota_\Pi Q)) = \det(Q)^2\det(\iota_\Pi^{\mathsf{T}}X^{-1}\iota_\Pi) = \det(\iota_\Pi^{\mathsf{T}}X^{-1}\iota_\Pi),
\]
so that \(\mathfrak{b}^{(k)}_{\Pi;\alpha,X}\) is well-defined.
\begin{exm}
    Let $\Pi = \partial_S(\mathbf{e}_1,\mathbf{e}_2)\subset \overline{\mathcal{X}_3}$, a boundary component of type $2$ consisting of matrices with vanishing third rows and columns. Let $\alpha = \mathbf{e}_1\otimes \mathbf{e}_1$, a component of type $1$ (i.e., a Satake point) on the boundary of $\Pi$. Then, for $X = (x^{ij})^{-1}$ and $Y = (y^{ij})^{-1}$, the type one Busemann-Selberg function is given by
    \[
    \mathfrak{b}^{(1)}_{\Pi;\alpha,X}(Y) = \frac{y^{11}/\sqrt{y^{11}y^{22} - (y^{12})^2}}{x^{11}/\sqrt{x^{11}x^{22} - (x^{12})^2}}.
    \]
\end{exm}
Busemann-Selberg functions can be expressed in terms of the classical Busemann functions $b_{\xi,X}$. In particular, when \(\alpha\) is a rank-one Satake point, the logarithm of a (type-\(k\)) Busemann-Selberg function decomposes as an explicit linear combination of the corresponding Busemann functions.
\begin{prop}\label{prop:3:2}
    Let $\Pi = \partial_S(V)$ be a boundary component of type $(n-k)\geq 2$, and $\alpha = \mathbf{v}\otimes \mathbf{v}$ be a Satake point on $\partial \Pi$ (so $\mathbf{v}\in V$). Denote by $\xi_{\mathbf{v}}$ and $\xi_V$ the corresponding vertices in the visual boundary $\partial_\infty \mathcal{X}_n$. Then for all \(X,Y\in\mathcal X_n\):
    \[
    \begin{split}
        & \log\mathfrak{b}_{\alpha,X}(Y) = \sqrt{\frac{n-1}{n}}b_{\xi_{\mathbf{v}},X}(Y), \\
        & \log\mathfrak{b}^{(k)}_{\Pi;\alpha,X}(Y) = \sqrt{\frac{n-1}{n}}b_{\xi_{\mathbf{v}},X}(Y) - \sqrt{\frac{k}{n(n-k)}}b_{\xi_V,X}(Y).
    \end{split}
    \]
\end{prop}
\begin{proof}
    By \(SL(n,\mathbb{R})\)-equivariance we may assume $\mathbf{v} = \mathbf{e}_1$ and $V = \mathrm{span}(\mathbf{e}_1,\dots,\mathbf{e}_{n-k})$. Explicit formulas in \cite{hattori1995collapsing} give
    \[
    b_{\xi_{\mathbf{v}},X}(Y) = \sqrt{\frac{n}{n-1}}\log \frac{Y^{-1}_{[1]}}{X^{-1}_{[1]}},\quad b_{\xi_V,X}(Y) = \sqrt{\frac{n}{k(n-k)}}\log \frac{Y^{-1}_{[n-k]}}{X^{-1}_{[n-k]}},
    \]
    where \(Y^{-1}_{[i]}\) denotes the $i$-th leading principal minor of $Y^{-1}$. One checks directly that these equalities coincide with the ratios defining \(\mathfrak b_{\alpha,X}\) and \(\mathfrak b^{(k)}_{\Pi,\alpha,X}\), yielding the claimed linear relations.
\end{proof}
Since higher-rank \(\alpha\) decompose as sums of rank-one matrices, every Busemann-Selberg function (of any type) can be written in terms of the original Busemann functions \(b_{\xi,X}\).

The main result in this subsection is the following $1$-Lipschitz continuity for Busemann-Selberg functions.
\begin{prop}\label{prop:3:4}
    Let $\Pi$ and $\alpha$ be as above, and any $X,Y_1,Y_2\in\mathcal{X}_n$. Then
    \[
    \lvert\log \mathfrak{b}^{(k)}_{\Pi;\alpha,X}(Y_1) - \log \mathfrak{b}^{(k)}_{\Pi;\alpha,X}(Y_2)\rvert\leq \sqrt{\frac{n-k-1}{n-k}}d(Y_1,Y_2).
    \]
\end{prop}
\begin{lem}\label{lem:3:3}
    The projection $\pi: \mathcal{X}_n\to \mathcal{X}_{n-k}$ from Definition \ref{def:3:5} is $1$-Lipschitz, i.e.,
    \[
    d(\pi(Y_1),\pi(Y_2))\leq d(Y_1,Y_2).
    \]
\end{lem}
% Implication of the lemma: if a line is asymptotic to Y, then the epsilon-neighborhood is asymptotic to the eplislon-neighborhood of Y.
% On one hand, for any point in the nbh of Y, we can construct a series of points within the epsilon nbh of the line and is asymptotic to it.
% On the other hand, if Y_i to Y, Z_i to Z, then \pi(Y_i) to Y and \pi(Z_i) to Z as well. Since the dist between \pi(Y_i) and \pi(Z_i) is not more than Y_i and Z_i - thus no more than epsilon, the dist Y to Z is no more than epsilon as well.
\begin{proof}
    Without loss of generality, take $Y_1=I_n$. Since $\pi$ is conjugation by an orthonormal-column matrix $\iota$, we have $\pi(Y_1) = \iota^{\mathsf{T}}\iota = I_{n-k}$. Hence it suffices to prove
    \[
    d\left(I_{n-k},\frac{\iota^{\mathsf{T}}Y\iota}{\det(\iota^{\mathsf{T}}Y\iota)^{1/(n-k)}}\right)\leq d(I_n,Y),
    \]
    where $Y = Y_2$. Let $\lambda_1\geq\dots\geq \lambda_n$ be the eigenvalues of $Y$, and let $\mu_1\geq\dots\geq\mu_{n-k}$ be those of $\iota^{\mathsf{T}}Y\iota$. By the Poincar\'e separation theorem,
    \[
    \lambda_i\geq \mu_i\geq \lambda_{i+k},\ i=1,\dots,n-k.
    \]
    Then, using Lemma~\ref{lem:append} in the Appendix and $\sum_i\log\lambda_i=0$, one obtains
    \[
    \begin{split}
        & d^2(I_{n-k},\iota^{\mathsf{T}}Y\iota/\det(\iota^{\mathsf{T}}Y\iota)^{1/(n-k)})\\
        & = \sum_{i=1}^{n-k}(\log \mu_i - \overline{\log \mu})^2\leq \sum_{i=1}^n(\log\lambda_i)^2 = d^2(I_n,Y),
    \end{split}
    \]
    which gives the desired bound.
\end{proof}
\begin{proof}[Proof of Proposition \ref{prop:3:4}]
    For $k=0$, the $1$-Lipschitz continuity for $\mathfrak{b}_{\alpha,X}$ with $\mathrm{rank}(\alpha) = 1$ follows from Proposition \ref{prop:3:2}. Any higher-rank $\alpha$ decomposes into rank-$1$ summands, so the result extends by linearity.
    
    For $k>0$, one reduces to the type-zero case in $\mathcal{X}_{n-k}$:
    \[
    \begin{split}
        & |\log \mathfrak{b}^{(k)}_{\Pi;\alpha,X}(Y_1) - \log \mathfrak{b}^{(k)}_{\Pi;\alpha,X}(Y_2)| \\
        & = \lvert \log \mathfrak{b}_{\pi(\alpha),\pi(X^{-1})^{-1}}(\pi(Y_1^{-1})^{-1}) - \log \mathfrak{b}_{\pi(\alpha),\pi(X^{-1})^{-1}}(\pi(Y_2^{-1})^{-1})\rvert \\
        & \leq \sqrt{\tfrac{n-k-1}{n-k}}d(\pi(Y_1^{-1})^{-1},\pi(Y_2^{-1})^{-1}).
    \end{split}
    \]
    Since $d(Y_1,Y_2)=d(Y_1^{-1},Y_2^{-1})$ and by Lemma~\ref{lem:3:3},
    \[
    d(\pi(Y_1^{-1})^{-1},\pi(Y_2^{-1})^{-1}) = d(\pi(Y_1^{-1}),\pi(Y_2^{-1}))\leq d(Y_1^{-1},Y_2^{-1}) = d(Y_1,Y_2),
    \]
    the proposition follows.
\end{proof}
We shall refer to the sublevel sets of the Busemann-Selberg functions as \textbf{horoballs}, and their level sets as \textbf{horospheres}.
\begin{defn}
    Let $\alpha\in \partial_S\mathcal{X}_n$ be a Satake boundary point and fix a reference point $X\in\mathcal{X}_n$. For each $r\in\mathbb{R}_+$, the \textbf{closed horoball} based at $\alpha$ with parameter $r$ is
    \[
    B(\alpha,r) = \{Y\in \mathcal{X}_n\mid \mathfrak{b}_{\alpha,X}(Y)\leq r\}.
    \]
    Replacing ``$\leq$'' by ``$<$'' yields the corresponding \textbf{open horoball}.

    The \textbf{horosphere} at level $r$ is the level set
    \[
    \Sigma(\alpha,r) = \{Y\in \mathcal{X}_n\mid \mathfrak{b}_{\alpha,X}(Y) = r\}.
    \]
\end{defn}
This construction generalizes to higher-type settings:
\begin{defn}
    Let $\Pi$ be a boundary component of type $n-k$, let $\alpha\in \partial \Pi$, and fix $X\in\mathcal{X}_n$. For each $r\in\mathbb{R}_+$, define the \textbf{$k$-th horoball} at $(\alpha,\Pi)$ by
    \[
    B_\Pi^{(k)}(\alpha,r) = \{Y\in \mathcal{X}_n\mid \mathfrak{b}_{\Pi;\alpha,X}^{(k)}(Y)\leq r\},
    \]
    and the corresponding \textbf{$k$-th horosphere} by
    Similarly, the  based at $(\Pi,\alpha)$ with parameter $r$ is defined as
    \[
    \Sigma_\Pi^{(k)}(\alpha,r) = \{Y\in \mathcal{X}_n\mid \mathfrak{b}_{\Pi;\alpha,X}^{(k)}(Y)= r\}.
    \]
\end{defn}
We illustrate these horospheres by restricting to the $2$-plane of diagonal matrices in $\overline{\mathcal X}_3$, with vertices~$\mathbf e_i\otimes\mathbf e_i$, $i=1,2,3$:
\begin{figure}[H]
    \begin{minipage}{0.32\textwidth}
        \centering
        \includegraphics[]{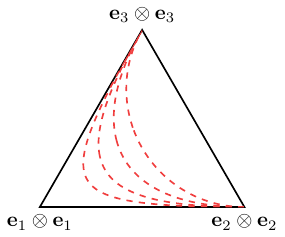}
    \end{minipage}
    \begin{minipage}{0.32\textwidth}
        \centering
        \includegraphics[]{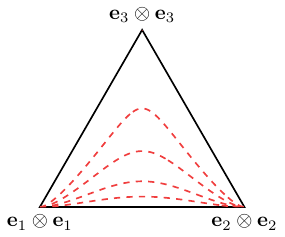}
    \end{minipage}
    \begin{minipage}{0.32\textwidth}
        \centering
        \includegraphics[]{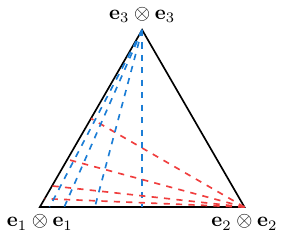}
    \end{minipage}
    \caption{Left: horospheres $\Sigma(\mathbf{e}_1\otimes\mathbf{e}_1,r)$ for varying $r$.\\ Center: horospheres $\Sigma(\mathbf{e}_1\otimes\mathbf{e}_1 + \mathbf{e}_2\otimes\mathbf{e}_2,r)$.\\ Right: type-$1$ horospheres $\Sigma^{(1)}_{\partial_S(\mathbf{e}_1,\mathbf{e}_2)}(\mathbf{e}_1\otimes\mathbf{e}_1,r)$ and $\Sigma^{(1)}_{\partial_S(\mathbf{e}_1,\mathbf{e}_3)}(\mathbf{e}_1\otimes\mathbf{e}_1,r)$ superimposed.}
\end{figure}
\subsection{Asymptotic Behavior of Busemann-Selberg Functions}\label{subsec:3:4}
In this subsection, we describe the asymptotic behavior of the Busemann-Selberg functions near the Satake boundary of~$\mathcal{X}_n$.

Recall that in hyperbolic geometry, the Busemann function $b_a(y)$ at an ideal point $a$ diverges to $+\infty$ whenever $y$ approaches any boundary point other than~$a$.  Analogous phenomena occur in the higher-rank symmetric space~$\mathcal{X}_n$.
\begin{lem}
    Let $\Pi\subset \overline{\mathcal{X}_n}$ be a boundary component of type $n-k$, pick $\alpha\in \partial\Pi$, and fix $X\in \mathcal{X}_n$. Suppose $\beta\in \partial_S \mathcal{X}_n$ satisfies
    \[
    \operatorname{Col}(\alpha)\backslash \operatorname{Col}(\beta)\neq \varnothing\text{ and } \operatorname{Col}(\Pi)\cap \operatorname{Col}(\beta) \neq \varnothing.
    \]
    Then for any $Y\in \mathcal{X}_n$,
    \[
    \lim_{\epsilon\to 0_+}\mathfrak{b}^{(k)}_{\Pi;\alpha,X}(\beta+\epsilon Y) = \infty.
    \]
\end{lem}
\begin{proof}
    After an $SL(n,\mathbb{R})$-action, we may assume
    \[
    \operatorname{Col}(\Pi) = \mathrm{span}\{\mathbf e_1,\dots,\mathbf e_{n-k}\}, 
    \quad
    \operatorname{Col}(\alpha) = \mathrm{span}\{\mathbf e_i : i\in\mathcal{A}\}, 
    \quad
    \operatorname{Col}(\beta)  = \mathrm{span}\{\mathbf e_j : j\in\mathcal{B}\},
    \]
    where $\mathcal{A}, \mathcal{B}\subset \{1,\dots,n\}$, with $\mathcal A\setminus\mathcal B\neq\varnothing$ but $\mathcal B\cap\{1,\dots,n-k\}\neq\varnothing$. 
    
    Write
    \[
    (\beta+\epsilon Y)^{-1}
    = M_{-1}\epsilon^{-1} + M_0 + O(\epsilon).
    \]
    The coefficient matrix $M_{-1}$ of the leading term is semi-positive definite, while its restriction to $\operatorname{Col}(\alpha)\backslash \operatorname{Col}(\beta)$ is positive definite. Hence
    \[
    \mathrm{tr}((\beta+\epsilon Y)^{-1}\alpha) = O(\epsilon^{-1})
    \]
    with a positive coefficient.
    
    On the other hand, let $\iota_\Pi = (\mathbf{e}_{1},\dots,\mathbf{e}_{(n-k)})$. Since $\operatorname{Col}(\Pi)\cap \operatorname{Col}(\beta) \neq \varnothing$, the principal $(n-k)\times (n-k)$-minor of $M_{-1}$ has at least one zero row and column. Therefore
    \[
    \det(\iota_\Pi^{\mathsf{T}}(\beta+\epsilon Y)^{-1}\iota_\Pi) = o(\epsilon^{-(n-k)}).
    \]
    Putting these estimates together,
    \[
    \mathfrak{b}^{(k)}_{\Pi;\alpha,X}(\beta+\epsilon Y) \varpropto \mathrm{tr}((\beta+\epsilon Y)^{-1}\alpha)\det (\iota_\Pi^{\mathsf{T}}(\beta+\epsilon Y)^{-1}\iota_\Pi)^{-1/(n-k)} \xrightarrow{\epsilon\to 0_+}\infty,
    \]
    as claimed.
\end{proof}
Differing from the hyperbolic case, the higher-type Busemann-Selberg functions on $\mathcal{X}_n$ can tend to zero as points approach certain boundary strata.
\begin{lem}\label{lem:3:5}
    Let $\Pi\subset \overline{\mathcal{X}_n}$ be a boundary component of type $n-k$, choose $\alpha\in \partial\Pi$, and fix $X\in \mathcal{X}_n$. Suppose $\beta\in \partial_S \mathcal{X}_n$ satisfied
    \[
    \operatorname{Col}(\alpha)\subset \operatorname{Col}(\beta)\text{ and }\operatorname{Col}(\Pi)\backslash \operatorname{Col}(\beta)\neq \varnothing.
    \]
    Then for any $Y\in \mathcal{X}_n$,
    \[
    \lim_{\epsilon\to 0_+}\mathfrak{b}^{(k)}_{\Pi;\alpha,X}(\beta+\epsilon Y) = 0.
    \]
\end{lem}
\begin{proof}
    After an $SL(n,\mathbb{R})$-action, assume
    \[
    \operatorname{Col}(\Pi) = \mathrm{span}(\mathbf{e}_{1},\dots,\mathbf{e}_{n-k}),\quad
    \operatorname{Col}(\alpha) = \mathrm{span}\{\mathbf e_i : i\in\mathcal{A}\}, \quad
    \operatorname{Col}(\beta)  = \mathrm{span}\{\mathbf e_j : j\in\mathcal{B}\},
    \]
    with $\mathcal A\subset\mathcal B$ but $\{1,\dots,n-k\}\setminus\mathcal B\neq\varnothing$.

    Writing
    \[
        (\beta+\epsilon Y)^{-1} = M_{-1}\epsilon^{-1} + M_0 + O(\epsilon),
    \]
    one sees that $M_{-1}$ vanishes on the columns indexed by $\mathcal{B}$. Since $\operatorname{Col}(\alpha)\subset \operatorname{Col}(\beta)$,
    \[
    \mathrm{tr}((\beta+\epsilon Y)^{-1}\alpha) = O(1),
    \]
    with a positive $\epsilon^0$ coefficient.

    Meanwhile, let $\iota_\Pi = (\mathbf{e}_{1},\dots,\mathbf{e}_{(n-k)})$. Because $\operatorname{Col}(\Pi)\backslash \operatorname{Col}(\beta)\neq \varnothing$, the principal $(n-k)\times(n-k)$-minor of $M_{-1}$ is nonzero and semi-positive definite, and the minor of $M_0$ is positive definite. Thus
    \[
    \det(\iota_\Pi^{\mathsf{T}}(\beta+\epsilon Y)^{-1}\iota_\Pi) = o(1).
    \]
    Therefore
    \[
    \mathfrak{b}^{(k)}_{\Pi;\alpha,X}(\beta+\epsilon Y)\varpropto\mathrm{tr}((\beta+\epsilon Y)^{-1}\alpha)\det (\iota_\Pi^{\mathsf{T}}(\beta+\epsilon Y)^{-1}\iota_\Pi)^{-1/(n-k)}\xrightarrow{\epsilon\to 0_+} 0,
    \]
    as desired.
\end{proof}
Two further asymptotic phenomena reveal the rich nature of Busemann-Selberg functions. The first arises when the point approaches the star of the corresponding boundary component.
\begin{lem}\label{lem:3:6}
    Let $\Pi\leq \Xi$ be boundary components of types $n-k\leq n-l$ in $\overline{\mathcal{X}_n}$, and let 
    \[
    \pi: \mathcal{X}_n\sqcup\mathrm{st}(\Xi)\to\mathcal{X}_{n-l}
    \]
    be the canonical projection. Pick $\alpha\in \partial \Pi$, $\beta\in \Xi$, and fix $X\in \mathcal{X}_n$. Then for each $Y\in\mathcal{X}_n$,
    \[
    \lim_{\epsilon\to 0_+}\mathfrak{b}^{(k)}_{\Pi;\alpha,X}(\beta+\epsilon Y) = \mathfrak{b}^{(k-l)}_{\pi(\Pi);\pi(\alpha),\pi(X^{-1})^{-1}}(\pi(\beta)).
    \]
    In particular, if $\Pi = \Xi$, the limit is a Busemann-Selberg function of type $0$.
\end{lem}
\begin{proof}
    Denote by
    \[
    \pi_1: \pi(\mathcal{X}_n\sqcup \mathrm{st}(\Pi)) = \mathcal{X}_{n-l}\sqcup\mathrm{st}(\pi(\Pi))\to \mathcal{X}_{n-k}
    \]
    the canonical projection, so that $\pi_1\circ\pi$ projects $\mathcal{X}_n\sqcup\mathrm{st}(\Pi)$ to $\mathcal{X}_{n-k}$. We first show
    \[
    \lim_{\epsilon\to 0_+}\pi((\beta+\epsilon Y)^{-1}) = \pi(\beta)^{-1}.
    \]
    Under the $SL(n,\mathbb{R})$-action, we assume
    \[
    \Xi = \partial_S(\mathbf{e}_{1},\dots,\mathbf{e}_{n-l}),\quad \beta = \mathrm{diag}(\beta_1,O),\quad Y = \left(\begin{array}{cc}
        Y_1 & Y_2 \\
        Y_2^{\mathsf{T}} & Y_3
    \end{array}\right),
    \]
    with $\beta_1,Y_1\in\mathrm{Sym}_{n-l}(\mathbb{R})$, $Y_3\in\mathrm{Sym}_{l}(\mathbb{R})$. Then
    \[
    (\beta+\epsilon Y)^{-1} = \left(\begin{array}{cc}
        \beta_1+\epsilon Y_1 & \epsilon  Y_2 \\
        \epsilon Y_2^{\mathsf{T}} & \epsilon Y_3
    \end{array}\right)^{-1}=\left(\begin{array}{cc}
        \beta_1^{-1}+O(\epsilon) & -\beta_1^{-1}Y_2Y_3^{-1}+O(\epsilon) \\
        -Y_3^{-1}Y_2^{\mathsf{T}}\beta_1^{-1}+O(\epsilon) & \epsilon^{-1} Y_3^{-1}+O(1)
    \end{array}\right).
    \]
    Hence,
    \[
    \lim_{\epsilon\to 0_+}\pi((\beta+\epsilon Y)^{-1}) = \frac{\beta_1^{-1}}{\det(\beta_1^{-1})^{1/(n-l)}} = \pi(\beta)^{-1}.
    \]
    It follows by continuity of trace and determinant that
    \[
    \begin{split}
        & \lim_{\epsilon\to 0_+}\mathfrak{b}^{(k)}_{\Pi;\alpha,X}(\beta+\epsilon Y) = \lim_{\epsilon\to 0_+}\frac{\mathrm{tr}(\pi_1(\pi((\beta+\epsilon Y)^{-1}))\pi_1(\pi(\alpha)))}{\mathrm{tr}(\pi_1(\pi(X^{-1}))\pi_1(\pi(\alpha)))} \\
        & = \lim_{\epsilon\to 0_+}\frac{\mathrm{tr}(\pi_1(\pi(\beta)^{-1})\pi_1(\pi(\alpha)))}{\mathrm{tr}(\pi_1(\pi(X^{-1}))\pi_1(\pi(\alpha)))} = \mathfrak{b}^{(k-l)}_{\pi(\Pi);\pi(\alpha),\pi(X^{-1})^{-1}}(\pi(\beta)).
    \end{split}
    \]
\end{proof}
\begin{exm}
    Let $\Pi = \partial_S(\mathbf{e}_1,\mathbf{e}_2)\subset \partial_S\mathcal{X}_3$, $X = I_3$, and $\alpha = \mathbf{e}_1\otimes \mathbf{e}_1$; let $X_0 = I_2$, and $\alpha_0 = \mathbf{e}_1\otimes \mathbf{e}_1\in \partial_\infty\mathbf{H}^2$. For each $\beta_0\in\mathcal{X}_2 = \mathbf{H}^2$ with $\beta = \mathrm{diag}(\beta_0,0)$, and for any $Y\in\mathcal{X}_3$,
    \[
    \lim_{\epsilon\to 0_+}\mathfrak{b}^{(1)}_{\Pi;\alpha,X}(\beta+\epsilon Y) = \mathfrak{b}_{\alpha_0,X_0}(\beta_0).
    \]
\end{exm}
In the second case, the Busemann-Selberg function diverges because its limit depends on the direction of approach to the boundary.
\begin{lem}\label{lem:3:7}
    Let $\Pi\leq\Xi$ be boundary components of types $n-k\leq n-l$ in $\overline{\mathcal{X}_n}$, and let
    \[
    \pi: \mathcal{X}_n\sqcup\mathrm{st}(\Xi)\to\mathcal{X}_{n-l}
    \]
    be the canonical projection. Let $\alpha\in \partial\Pi$, $\beta\in \partial_S \mathcal{X}_n$, and fix $X\in\mathcal{X}_n$ satisfying $\operatorname{Col}(\beta)\oplus \operatorname{Col}(\Xi) = \mathbb{R}^n$. Then for every $Y\in \mathcal{X}_n$,
    \[
    \lim_{\epsilon\to 0_+}\mathfrak{b}^{(k)}_{\Pi;\alpha,X}(\beta+\epsilon Y) = \mathfrak{b}_{\pi(\Pi);\pi(\alpha),\pi(X^{-1})^{-1}}^{(k-l)}(\pi(Y)).
    \]
    In particular, if $\Pi = \Xi$, the path limit is a Busemann-Selberg function of type $0$.
\end{lem}
\begin{proof}
    Conjugate so that
    \[
    \Xi = \partial_S(\mathbf{e}_{1},\dots,\mathbf{e}_{n-l}),\ \beta = \mathrm{diag}(O,\beta_3),\ Y = \left(\begin{array}{cc}
        Y_1 & Y_2 \\
        Y_2^{\mathsf{T}} & Y_3
    \end{array}\right),
    \]
    where $\beta_3,Y_3\in GL(l,\mathbb{R})$ and $Y_1\in GL(n-l,\mathbb{R})$. The block-matrix inversion shows
    \[
    (\beta+\epsilon Y)^{-1} = \left(\begin{array}{cc}
        \epsilon Y_1 & \epsilon  Y_2 \\
        \epsilon Y_2^{\mathsf{T}} & \beta_3+\epsilon Y_3
    \end{array}\right)^{-1}=\left(\begin{array}{cc}
        \epsilon^{-1}Y_1^{-1} + O(1) & -Y_1^{-1}Y_2\beta_3^{-1}+O(\epsilon)\\ -Y_1^{-1}Y_2^{\mathsf{T}}\beta_3^{-1}+O(\epsilon) & \beta_3^{-1}+O(\epsilon)
    \end{array}\right).
    \]
    Hence
    \[
    \lim_{\epsilon\to 0_+}\pi((\beta + \epsilon Y)^{-1}) = \pi(Y)^{-1}.
    \]
    The remainder of the argument follows exactly as in Lemma \ref{lem:3:6}, by continuity of trace and determinant in the definition of $\mathfrak{b}^{(k)}_{\Pi;\alpha,X}$.
\end{proof}
\begin{exm}
    Let $\Pi = \partial_S(\mathbf{e}_1,\mathbf{e}_2)\subset \overline{\mathcal{X}_3}$, $X = I_3$, $\alpha = \mathbf{e}_1\otimes \mathbf{e}_1$, and $\beta = \mathbf{e}_3\otimes \mathbf{e}_3$. Let $X_0 = I_2$ and $\alpha_0 = \mathbf{e}_1\otimes \mathbf{e}_1\in \overline{\mathbf{H}^2}$. Then for any $Y\in\mathcal{X}_3$, with $Y_0\in\mathbf{H}^2$ being its projection to the first two rows and columns, we have
    \[
    \lim_{\epsilon\to 0_+}\mathfrak{b}^{(1)}_{\Pi;\alpha,X}(\beta+\epsilon Y) = \mathfrak{b}_{\alpha_0,X_0}(Y_0).
    \]
\end{exm}
To conclude, we summarize the behavior of $\mathfrak{b}_{\Pi;\alpha,X}^{(k)}$ to the Satake boundary.
\begin{table}[H]
    \centering
    \resizebox{\columnwidth}{!}{
    \begin{tabular}{c|c|c|c|c}
        \hline
        \multirow{2}{*}{Conditions} & \multicolumn{2}{c|}{$\operatorname{Col}(\alpha)\subseteq \operatorname{Col}(\beta)$} & \multicolumn{2}{c}{$\operatorname{Col}(\alpha)\backslash \operatorname{Col}(\beta)\neq\varnothing$}\\
        \cline{2-5}
        & $\operatorname{Col}(\Pi)\backslash \operatorname{Col}(\beta)\neq\varnothing$ & $\operatorname{Col}(\Pi)\subseteq \operatorname{Col}(\beta)$ & $\operatorname{Col}(\beta)\cap\operatorname{Col}(\Pi)=\varnothing$ & $\operatorname{Col}(\beta)\cap\operatorname{Col}(\Pi)\neq\varnothing$\\
        \hline
        \begin{tabular}[c]{@{}c@{}}$\lim_{\epsilon\to 0_+}\mathfrak{b}^{(k)}_{\Pi;\alpha,X}$\\ $(\beta+\epsilon Y)$\end{tabular} & $0$ & $\mathfrak{b}_{\pi(\alpha),\pi(X^{-1})^{-1}}(\pi(\beta))$ & $\mathfrak{b}_{\pi(\alpha),\pi(X^{-1})^{-1}}(\pi(Y))$ & $\infty$ \\
        \hline
    \end{tabular}
    }
\end{table}
% \begin{table}[H]
%     \centering
%     \resizebox{\columnwidth}{!}{
%     \begin{tabular}{c|c|c|c|c}
%         \hline
%         \multirow{2}{*}{Conditions} & \multicolumn{2}{c|}{$\operatorname{Nul}(\beta)\subset \operatorname{Nul}(\alpha)$} & \multicolumn{2}{c}{$\operatorname{Nul}(\beta)\backslash \operatorname{Nul}(\alpha)\neq\varnothing$}\\
%         \cline{2-5}
%         & $\operatorname{Nul}(\beta)\backslash \operatorname{Nul}(\Pi)\neq\varnothing$ & $\operatorname{Nul}(\beta)\subset \operatorname{Nul}(\Pi)$ & \begin{tabular}[c]{@{}c@{}}$\mathrm{span}(\operatorname{Nul}(\beta),\operatorname{Nul}(\Pi))$\\ $=\mathbb{R}^n$\end{tabular} & \begin{tabular}[c]{@{}c@{}}$\mathrm{span}(\operatorname{Nul}(\beta),\operatorname{Nul}(\Pi))$\\ $\neq\mathbb{R}^n$\end{tabular}\\
%         \hline
%         \begin{tabular}[c]{@{}c@{}}$\lim_{\epsilon\to 0_+}\mathfrak{b}^{(k)}_{\Pi;\alpha,X}$\\ $(\beta+\epsilon Y)$\end{tabular} & $0$ & $\mathfrak{b}_{\pi(\alpha),\pi(X^{-1})^{-1}}(\pi(\beta))$ & $\mathfrak{b}_{\pi(\alpha),\pi(X^{-1})^{-1}}(\pi(Y))$ & $\infty$ \\
%         \hline
%     \end{tabular}
%     }
% \end{table}
\subsection{Finite Volume Convex Polytopes in \texorpdfstring{$\mathcal{X}_n$}{Lg}}
A \textbf{convex polytope} $D\subset\mathcal X_n$ is by definition the intersection of finitely many affine half-spaces in $\mathrm{Sym}_n(\mathbb{R})$ with the hypersurface $\mathcal{X}_{n,\mathrm{hyp}}$. Equivalently, one may view
\[
D = \mathbf{D}\cap \mathcal{X}_{n,\mathrm{proj}},
\]
where $\mathbf D\subset\mathbf{P}(\mathrm{Sym}_n(\mathbb{R}))$ is a projective convex polytope with finitely many faces.

In Proposition \ref{lem:append:2} in the Appendix, we show that $D$ has finite volume (with respect to the Riemannian metric on $\mathcal{X}_n$) if and only if its corresponding projective polytope $\mathbf{D}$ lies entirely inside the Satake compactification~$\overline{\mathcal X_n}$. We therefore adopt the following equivalent criterion:
\begin{defn}\label{defn:3:1}
    A convex polytope $D\subset\mathcal{X}_n$ is said to have \textbf{finite volume} if there exists a projective polytope $\mathbf{D}\subset\overline{\mathcal X_n}\subset \mathbf{P}(\mathrm{Sym}_n(\mathbb{R}))$ such that
    \[
        D = \mathbf{D}\cap \mathcal{X}_n
    \]
    In this case, $\mathbf{D}$ is the \textbf{Satake compactification} of~$D$, denoted $\overline D=\mathbf D$.
    
    The \textbf{Satake boundary} of $D$ is then
    \[
    \partial_S D = \overline{D}\cap \partial_S D.
    \]
\end{defn}
Since $\partial_S\mathcal{X}_n$ decomposes into boundary components indexed by subspaces of $\mathbb{R}^n$, the same holds for $\partial_S D$:
\begin{defn}
    Let $D\subset\mathcal{X}_n$ be a finitely-sided convex polytope of finite volume. For each linear subspace $V\subset \mathbb{R}^n$, define the \textbf{Satake boundary component}
    \[
    \Phi_V = \partial_S D\cap \partial_S(V),
    \]
    where $\partial_S(V)$ is the corresponding component of $\partial_S\mathcal{X}_n$. The integer $k = \dim V$ is called the \textbf{type} of $\Phi_V$.
\end{defn}
It is immediate that the closure of a boundary component decomposes into smaller strata:
\[
\overline{\Phi_V} = \bigsqcup_{W\subset V} \Phi_W.
\]
Furthermore, the Satake boundary of any finite-volume, finitely-sided polytope admits a natural combinatorial description:
\begin{prop}
    Let \(D\subset\mathcal X_n\) be a finitely-sided convex polytope of finite volume, with Satake compactification \(\overline D\). Then for each nonempty boundary component \(\Phi_V\subset\partial_S D\), its closure \(\overline{\Phi_V}\) is a face of the projective polytope \(\overline D\).
\end{prop}
\begin{proof}
    Write
    \[
    \overline{D} = \mathrm{conv}\{\alpha_1,\dots,\alpha_m\},
    \]
    where each vertex $\alpha_i\in \overline{\mathcal{X}_n}$. Let $V = \subset \mathbb{R}^n$ be a subspace such that $\Phi_V = \partial_S D\cap \partial_S(V)$ is non-empty, and
    \[
    I = \{i\mid \alpha_i\in \overline{\partial_S(V)}\}.
    \]
    We claim:
    \begin{itemize}
        \item The convex hull $\mathrm{conv}(\{\alpha_i\}_{i\in I})$ is a face of $\overline{D}$.
        \item This convex hull coincides with $\overline{\Phi_V}$.
    \end{itemize}
    To prove the first claim, note that \(\alpha\in\overline{\partial_S(V)}\) if and only if the associated bilinear form vanishes on \(V^\perp\). For each \(j\notin I\), the kernel of \(\alpha_j\) in \(V^\perp\) is a proper Zariski-closed subset, so we can choose \(\mathbf u\in V^\perp\) and \(\epsilon>0\) such that
    \[
        \mathbf{u}^{\mathsf T}\alpha_i\,\mathbf{u} = 0 \quad (i\in I),
        \quad\text{and}\quad
        \mathbf{u}^{\mathsf T}\alpha_j\,\mathbf{u} \ge \epsilon \quad (j\notin I).
    \] 
    The hyperplane $\{\alpha\mid \mathbf{u}^{\mathsf T}\alpha\,\mathbf{u} = 0\}$ separates $\mathrm{conv}(\{\alpha_i\}_{i\in I})$ from $\mathrm{conv}(\{\alpha_j\}_{j\notin I})$, proving that the former is indeed a face.

    For the second claim, observe that $\overline{\Phi_V} = \overline{\partial_S(V)}\cap \overline{D}$ is convex and contains all \(\alpha_i\) with \(i\in I\), so \(\mathrm{conv}(\{\alpha_i\}_{i\in I})\subseteq\overline{\Phi_V}\). Conversely, any point of \(\overline{\Phi_V}\) is a convex combination of the vertices \(\alpha_1,\dots,\alpha_m\), but the above separation also applies to $\overline{\Phi_V}$, implying that vertices with \(j\notin I\) cannot appear in such combinations.  Therefore $\overline{\Phi_V} = \mathrm{conv}(\{\alpha_i\}_{i\in I})$, as claimed.
\end{proof}
From this it follows:
\begin{cor}\label{col:3:1}
    Let \(D\subset\mathcal X_n\) be a finitely-sided, finite-volume convex polytope. Then:
    \begin{itemize}
        \item There are only finitely many subspaces \(V\subset\mathbb{R}^n\) for which \(\Phi_V\neq\varnothing\). Equivalently,
        \[
        \partial_S D = \bigsqcup_{V\in\mathcal{V}} \Phi_V,
        \]
        is a finite disjoint union.
        \item For each nonempty boundary component $\Phi_V$ of type $k = \dim(V)$, its image under the canonical identification $\pi_V: \partial_S(V)\to \mathcal{X}_k$ is again a finitely-sided convex polytope of finite volume in~\(\mathcal X_k\).
        \item The Satake boundary of \(\pi_V(\Phi_V)\) decomposes as
        \[
        \partial_S\pi_V(\Phi_V) = \pi_V\left(\bigsqcup_{W\in\mathcal{V},\ W\subsetneq V} \Phi_W\right).
        \]
    \end{itemize}
\end{cor}
We call any face of a boundary component \(\Phi\subset\partial_S D\) (including \(\Phi\) itself) a \textbf{Satake face} of \(D\), and denote the set of all Satake faces by \(\mathcal F_S(D)\).
\vspace{12pt}
\section{Preliminary Lemmas for the Main Theorem}\label{Sec:4}
Let $D$ be an exact hyperbolic Dirichlet domain satisfying the tiling condition, and let $M=D/\sim$ be the associated quotient manifold (or orbifold) by gluing up the facets. The completeness of $M$ follows from two facts:
\begin{itemize}
    \item Balls centered in the \textbf{thick part} of $M$ are compact up to the injectivity radius.
    \item Lemma~\ref{lem:1:1} implies that the \textbf{thin part} of $M$ consists solely of cusps, which likewise admit compact neighborhoods.
\end{itemize}
In this section, we show that an analogous structure holds for the thin part of a finite-volume Dirichlet-Selberg quotient in~$\mathcal X_n$.
\subsection{Tangency of Horospheres to the Satake Boundary}
In hyperbolic space, any horosphere based at an ideal point $a\in\mathbf{H}^n$ meets the visual boundary only at~$a$, and does so tangentially. A similar tangency phenomenon occurs for horospheres in~$\mathcal X_n$, with additional cases to consider.
\begin{prop}\label{prop:4:1}
    Let $\alpha\in\partial_S\mathcal X_n$ lie in the boundary component~$\Pi$, and fix $r>0$. Denote the closed horoball and its boundary by
    \[
        B = B(\alpha,r), \quad \Sigma = \Sigma(\alpha,r), \quad \overline\Sigma = \partial\overline{B}.
    \]
    Then:
    \begin{itemize}
        \item The horosphere meets the Satake boundary exactly along the closure of the star of~$\Pi$:
        \[
        \overline{\Sigma}\cap \partial_S\mathcal{X}_n = \overline{\mathrm{st}(\Pi)}.
        \]
        \item For each Satake point $\beta\in \mathrm{st}(\Pi)$, the hypersurfaces $\overline{\Sigma}$ and $\partial_S\mathcal{X}_n$ are tangent at $\beta$.
    \end{itemize}
\end{prop}
\begin{proof}
    \textbf{Boundary contact}. Let $\beta\in\partial_S\mathcal X_n$. By the asymptotic lemmas of Subsection~\ref{subsec:3:4}:
    \begin{itemize}
        \item If $\beta\in\mathrm{st}(\Pi)$, then $\operatorname{Col}(\beta)\supseteq \operatorname{Col}(\alpha)$, so 
        $\mathfrak b_{\alpha}(Y)\to0$ as $Y\to\beta$.  Hence $\beta\in\overline{B(\alpha,r)}$.
        \item If $\beta\notin\overline{\mathrm{st}(\Pi)}$, then in any neighborhood of $\beta$, $\mathfrak b_{\alpha}(Y)\to+\infty$ along any approach direction, so $\beta\notin\overline{B(\alpha,r)}$.
    \end{itemize}
    This shows
    \[
    \overline{\Sigma}\cap \partial_S\mathcal{X}_n = \overline{B}\cap \partial_S\mathcal{X}_n = \overline{\mathrm{st}(\Pi)}.
    \]
    \textbf{Tangency}. Fix $\beta\in\mathrm{st}(\Pi)$. Conjugate so that
    \[
    \beta = \begin{pmatrix}\beta_0 & 0\\0 & 0\end{pmatrix},
    \quad
    \beta_0\in GL({n-l},\mathbb{R}).
    \]
    Write a general tangent direction at $\beta$ in projective coordinates as
    \[
        A = \begin{pmatrix}A_1 & A_2^{\mathsf T}\\A_2&A_3\end{pmatrix},\quad A_3\in \mathrm{Mat}_l(\mathbb{R}).
    \]
    With reference point $X=I_n$, the horosphere $\Sigma(\alpha,r)$ is cut out (in projective space) by
    \[
        f(Y) = \bigl(\mathrm{tr}(\alpha Y)\bigr)^n - r^n\bigl(\det Y\bigr)^{\,n-1}=0.
    \]
    Expanding $f(\beta+tA)$ for small $t$, the dominant term in  $r^n(\det(\beta + tA))^{n-1}$ is 
    \[
        r^n(\det(\beta_0)\det(A_3))^{n-1}t^{(n-1)l},
    \]
    while $\bigl(\mathrm{tr}(\alpha(\beta+tA))\bigr)^n$ has strictly higher order in~$t$. Thus $A$ lies in the tangent cone to $\overline\Sigma$ if and only if
    \[
        \det(A_3)=0.
    \]
    On the other hand, the Satake boundary $\partial_S\mathcal X_n$ in projective coordinates is the hypersurface $\det Y=0$. Its linearization at $\beta$ likewise vanishes on exactly those $A$ for which $\det(A_3)=0$.
    
    Hence at each $\beta\in\mathrm{st}(\Pi)$, the two hypersurfaces $\overline\Sigma$ and $\partial_S\mathcal X_n$ share the same tangent cone, proving they are tangent.
\end{proof}
We have a more generalized tangency property for higher-type cases.
\begin{prop}\label{prop:4:2}
    Let $\Xi\subset\overline{\mathcal X_n}$ be a boundary component of type $n-k$, and let $\Pi<\Xi$ be a smaller boundary component containing a Satake point $\alpha$. For each $r>0$, denote the $k$-th horosphere and its boundary by
    \[
        B = B^{(k)}_{\Xi}(\alpha,r), \quad \Sigma = \Sigma^{(k)}_{\Xi}(\alpha,r), \quad \overline\Sigma = \partial\overline{B}.
    \]
    Then:
    \begin{itemize}
        \item The horosphere meets the Satake boundary at:
        \[
        \overline\Sigma \cap\partial_S\mathcal X_n = \overline{\mathrm{st}(\Pi) \setminus \Bigl(\!\bigsqcup_{\Xi_0\ge\Xi}\!\bigl(\Xi_0\setminus B^{(k-l)}_{\Xi}(\alpha,r)\bigr)\Bigr)},
        \]
        where each $\Xi_0\ge\Xi$ is a boundary component of type $n-l$ and $B^{(k-l)}_{\Xi}(\alpha,r)$ is the corresponding $(k-l)$-th horoball.
        \item For each Satake point 
        \[
            \beta\in \mathrm{st}(\Pi) \setminus \Bigl(\!\bigsqcup_{\Xi_0\ge\Xi}\!\bigl(\Xi_0\setminus B^{(k-l)}_{\Xi}(\alpha,r)\bigr)\Bigr),
        \]
        the hypersurfaces $\overline\Sigma$ and $\partial_S\mathcal X_n$ are tangent at~$\beta$.
    \end{itemize}
\end{prop}
\begin{proof}
    \textbf{Boundary contact}. By the asymptotic lemmas of Subsection~\ref{subsec:3:4}, any Satake point $\beta$ falls into exactly one of the following cases, determining whether $\beta\in\overline B$:
    \begin{itemize}
        \item If $\beta\in\mathrm{st}(\Pi)\setminus\mathrm{st}(\Xi)$, then $\operatorname{Col}(\beta)\supseteq\operatorname{Col}(\alpha)$ and $\operatorname{Col}(\Xi)\not\subset\operatorname{Col}(\beta)$, so $\mathfrak b^{(k)}_{\Xi;\alpha}(Y)\to0$ as $Y\to\beta$, and thus $\beta\in B$.
        \item If $\beta\in\mathrm{st}(\Xi)$, then $\operatorname{Col}(\beta)\supseteq\operatorname{Col}(\Xi)$ and $\mathfrak b^{(k)}_{\Xi;\alpha}(Y)\to\mathfrak b^{(k-l)}_{\pi(\Xi);\pi(\alpha)}(\pi(\beta))$ as $Y\to\beta$, so $\beta\in B$ precisely when that limit is $\leq r$.
        \item If $\beta$ is not in the closure of previous cases, $\mathfrak b^{(k)}_{\Xi;\alpha}(Y)\to+\infty$ along any approach, so $\beta\notin\overline B$.
    \end{itemize}
    To see tangency, fix any such $\beta\in\overline\Sigma\cap\partial_S\mathcal X_n$. Similar to the previous lemma, a tangent vector $A\in T_\beta\mathbf{P}(\mathrm{Sym}_n(\mathbb{R}))$ lies in the tangent cone of $\overline\Sigma$ if and only if it lies in that of either the equation $ \{Y\mid \det(\pi_{\Xi}(Y))=0\}$ or the boundary defining inequality of $\mathcal X_n$. But the latter hypersurface entirely contains $\mathcal{X}_n$, making the two tangent cones coincide. This proves the tangency of $\overline{\Sigma}$ and $\partial_S\mathcal{X}_n$.
\end{proof} 
\subsection{Satake Face Cycles}\label{subsec:4:2}
In hyperbolic geometry, a finite-volume manifold $M$ is complete precisely when each cusp link $L[a]$ is a Euclidean \textbf{isometry} manifold, i.e.\ the holonomy similarity transformation of every generator of $\pi_1(L[a])$ lies in the Euclidean isometry group \cite{ratcliffe1994foundations, goldman2022geometric}. This condition is satisfied by Dirichlet domain quotients, where each ideal cycle preserves the corresponding Busemann function $b_a$ \cite{kapovich2023geometric}.

We generalize this to $\mathcal{X}_n$ by defining \textbf{cycles} of Satake faces and proving they preserve Busemann-Selberg functions.
\begin{defn}
    Let $D\subset\mathcal X_n$ be a finite-volume, finitely-sided polytope.  Denote by \(\mathcal F(D)\) its set of (ordinary) faces, and by \(\mathcal F_S(D)\) its set of Satake faces, each Satake face \(\Phi\) lying in a boundary component~\(\Pi\).
    \begin{itemize}
        \item We say a Satake face \(\Phi\in\mathcal F_S(D)\) is \textbf{incident with} a face \(F\in\mathcal F(D)\) if \(\Phi\subset\overline F\).
        \item More precisely, the pair \((\Phi,\Pi)\) is \textbf{incident with} \(F\) if \(\overline\Phi\subseteq \overline F\cap \overline{\Pi}\), and it is \emph{precisely incident} if \(\overline\Phi = \overline F\cap \overline{\Pi}\).
        \item A \textbf{pairing} of two Satake faces \(\Phi,\Phi'\in\mathcal F_S(D)\) is given by a facet-pairing isometry \(g_F\) so that
        \[
        \Phi\subset\overline F, \quad \Phi'\subset\overline{F'}\quad\bigl(F'=g_F^{-1}F\bigr), \quad g_F^{-1}.\Phi =\Phi'.
        \]
        We write \([\Phi]\) for the equivalence class of \(\Phi\) under such pairings.
        \item A \textbf{cycle} of the Satake face \(\Phi\) is a finite sequence \(\{\Phi_0,\Phi_1,\dots,\Phi_m\}\) of faces in \([\Phi]\) with \(\Phi_0=\Phi_m=\Phi\), and isometries \(g_i\) so that \(\Phi_i=g_i.\Phi_{i-1}\) for \(i=1,\dots,m\). \ The product
        \[
            w \;=\; g_1\,g_2\cdots g_m\in SL(n,\mathbb{R})
        \]
        is called the \textbf{word} of the cycle.
    \end{itemize}
\end{defn}
Below is our generalized preservation property for usual Busemann-Selberg functions under Satake face cycles. 
\begin{prop}\label{prop:4:3}
    Let \(D = DS(X,\Gamma_0)\subset\mathcal X_n\) be a Dirichlet-Selberg domain satisfying the hypotheses of Theorem~\ref{thm:1:2}, and let \(\Phi\) be a Satake face of type \(n-k\). Suppose \(\{\Phi_0,\Phi_1,\dots,\Phi_m\}\) is a cycle of \(\Phi\) with associated word 
    \[
        w = g_1g_2\cdots g_m\in SL(n,\mathbb{R}).
    \]
    Then:
    \begin{itemize}
        \item The action of \(w\) on the boundary component \(\mathrm{span}(\Phi)\) has finite order.
        \item There exists a Satake point \(\alpha_\Phi\) in the relative interior of \(\Phi\) such that \(w.\alpha_\Phi = \alpha_\Phi\).
        \item For every \(Y,Z\in\mathcal X_n\),
        \[
            \mathfrak{b}_{\alpha_{\Phi},Z}(Y) = \mathfrak{b}_{\alpha_{\Phi},Z}(w.Y).
        \]
    \end{itemize}
\end{prop}
Proposition~\ref{prop:4:3} rests on the following equivariance lemma.
\begin{lem}\label{lem:4:1}
    Let $g\in SL(n,\mathbb{R})$, fix $X\in\mathcal{X}_n$, and let $\alpha\in \partial_S \mathrm{Bis}(X,g^{-1}.X)$. Then:
    \begin{enumerate}
    \item $\mathrm{tr}(X^{-1}\alpha) = \mathrm{tr}(X^{-1}(g.\alpha))$.
    \item For all $Y\in\mathcal{X}_n$,
    \[
      \mathfrak b_{\alpha,X}(Y) = \mathfrak b_{g.\alpha,X}\bigl(g. Y\bigr).
    \]
    \end{enumerate}
\end{lem}
\begin{proof}
    Since \(\alpha\) lies in the Satake boundary of the bisector \(\mathrm{Bis}(X,g^{-1}.X)\), one has
    \[
        \mathrm{tr}(X^{-1}\alpha) = \mathrm{tr}((g^{-1}.X)^{-1}\alpha) = \mathrm{tr}(gX^{-1}g^\mathrm{T}\alpha) = \mathrm{tr}(X^{-1}(g.\alpha)),
    \]
    proving (1).

    For (2), note
    \[
    \mathrm{tr}((g.Y)^{-1}(g.\alpha)) = \mathrm{tr}(g^{-1}Y^{-1}(g^{-1})^\mathrm{T}g^{\mathsf{T}}\alpha g) = \mathrm{tr}(g^{-1}Y^{-1}\alpha g) = \mathrm{tr}(Y^{-1}\alpha).
    \]
    Hence
    \[
    \mathfrak{b}_{\alpha,X}(Y) = \frac{\mathrm{tr}(Y^{-1}(\alpha))}{\mathrm{tr}(X^{-1}(\alpha))} = \frac{\mathrm{tr}((g.Y)^{-1}(g.\alpha))}{\mathrm{tr}(X^{-1}(g.\alpha))} = \mathfrak{b}_{g.\alpha,X}(g.Y).
    \]
\end{proof}
\begin{proof}[Proof of Proposition \ref{prop:4:3}]
    Let \(\{\Phi_0,\Phi_1,\dots,\Phi_m\}\) be a cycle of the Satake face \(\Phi\), with \(\Phi_i = g_i\cdot\Phi_{i-1}\) for \(i=1,\dots,m\) and \(\Phi_0=\Phi_m=\Phi\). For any interior point \(\xi\in\mathrm{span}(\Phi)\), set \(\xi_0=\xi\), \(\xi_i=g_i\cdot\xi_{i-1}\) for \(i=1,\dots,m\).

    Since each \(g_i\) pairs facets of the Dirichlet-Selberg domain, we have 
    \(\xi_{i-1}\in \Phi_{i-1}\subset \overline{\mathrm{Bis}(X, g_i^{-1}.X)}\). By Lemma~\ref{lem:4:1},
    \[
        \mathrm{tr}(X^{-1}\xi_{i-1}) = \mathrm{tr}(X^{-1}(g_i.\xi_{i-1})) = \mathrm{tr}(X^{-1}\xi_i).
    \]
    Iterating gives
    \begin{equation}\label{equ:4:1}
    \mathrm{tr}(X^{-1}\xi) = \mathrm{tr}(X^{-1}(w.\xi)),\quad
    w=g_1\cdots g_m.
    \end{equation}
    \textbf{Finite-order on the boundary component}. Conjugate so that $\mathrm{span}(\Phi) = \partial_S(\mathbf{e}_{1},\dots,\mathbf{e}_{n-k})$, and let $\pi: \mathcal{X}_n\sqcup \mathrm{st}(\mathrm{span}(\Phi))\to \mathcal{X}_{n-k}$ be the projection dropping the last $k$ coordinates (with determinant normalization). Then $w$ preserves $\mathrm{span}(\Phi)$, and its restriction $\pi(w)\in GL(n-k,\mathbb{R})$ is a nonzero multiple of an $\mathcal{X}_{n-k}$-isometry.

    Define
    \[
        s:\mathrm{span}(\Phi)\to \mathbb{R},\quad  s(\xi) = \frac{\mathrm{tr}(X^{-1}\xi)}{\det(\iota_\Phi^{\mathsf{T}}\xi \iota_\Phi)^{1/(n-k)}},
    \]
    where $\iota_\Phi$ is the \(n\times(n-k)\) matrix selecting the first \(n-k\) coordinates. Then by \eqref{equ:4:1},    
    \[
    \begin{split}
        & s(w.\xi) = \frac{\mathrm{tr}(X^{-1}(w.\xi))}{\det(\pi(w).(W^{\mathsf{T}}\xi W))^{1/(n-k)}} \\
        & = \frac{\mathrm{tr}(X^{-1}\xi)}{\lvert\det(\pi(w))\rvert^{2/(n-k)}\det(W^{\mathsf{T}}\xi W)^{1/(n-k)}} = \frac{s(\xi)}{\lvert\det(\pi(w))\rvert^{2/(n-k)}}.
    \end{split}
    \]
    On the other hand, $s$ attains a unique minimum at $\alpha = \mathrm{diag}(\pi(X^{-1})^{-1},O_k)$. Uniqueness forces $w.\alpha = \alpha$ and $\lvert\det(\pi(w))\rvert = 1$. Hence \(\pi(w)\) lies in the compact subgroup \(\mathrm{O}(n-k)\) and, because it preserves the polytope \(\pi(\Phi)\), has finite order.
    % \[
    % s(\xi) = \det(W^{\mathsf{T}}X^{-1}W)^{1/(n-k)}s(\pi(X^{-1})^{-1},\pi(\xi)),
    % \]
    % thus it has a unique minimum at
    % \[
    % \alpha: = diag(\pi(X^{-1})^{-1},O).
    % \]

    \textbf{Existence of a fixed Satake point}. If \((\pi(w))^l = I_{n-k}\) for some \(l>0\), then the barycenter of the orbit \(\{\xi,\pi(w)\xi,\dots,\pi(w)^{l-1}\xi\}\) is a \(\pi(w)\)-fixed point in the interior of \(\Phi\). Lifted back to \(\mathcal X_n\), this yields the desired \(\alpha_\Phi\).

    \textbf{Preservation of the Busemann-Selberg function}. Write \(\alpha_0=\alpha_\Phi\) and \(\alpha_i=g_i.\alpha_{i-1}\). By Lemma~\ref{lem:4:1}, for every \(Y\in\mathcal X_n\),
    \[
    \mathfrak{b}_{\alpha_{i-1},X}(Y) = \mathfrak{b}_{g_i.\alpha_{i-1},X}(g_i.X) = \mathfrak{b}_{\alpha_i,X}(g_i.Y).
    \]
    Iterating from $i=1$ to $m$ and using $\alpha_m = \alpha_0$ gives
    \[
    \mathfrak{b}_{\alpha_{\Phi},X}(Y) = \mathfrak{b}_{\alpha_{\Phi},X}(w.Y),
    \]
    and replacing $X$ by any $Z\in \mathcal{X}_n$ preserves the equality.
\end{proof}
A similar preservation property holds for higher-type Busemann-Selberg functions as well.
\begin{prop}\label{prop:4:4}
    Let $D = DS(X,\Gamma_0)\subset\mathcal{X}_n$ be a Dirichlet-Selberg domain satisfying the hypotheses of Theorem \ref{thm:1:2}. Let $\Phi$ be a Satake face of type $n-k$, and $\Psi$ be a Satake face of type $n-l$ with $l<k$ such that $\Psi>\Phi$. Denote by $\Pi = \mathrm{span}(\Psi)$ the corresponding boundary component.
    
    If $w$ is the word of a common cycle of both $\Phi$ and $\Psi$, and if $\alpha_{\Phi}\in\Phi$ is a $w$-fixed interior point (cf. Proposition~\ref{prop:4:3}), then for all $Y,Z\in\mathcal{X}_n$,
    \[  
        \mathfrak{b}^{(l)}_{\Pi;\alpha_{\Phi},Z}(Y) = \mathfrak{b}^{(l)}_{\Pi;\alpha_{\Phi},Z}(w.Y).
    \]
\end{prop}
\begin{proof}
    Recall from \eqref{equ:3:2} that
    \[
    \mathfrak b^{(l)}_{\Pi;\alpha,Z}(Y)
    = \mathfrak b_{\alpha,Z}(Y)
      \det\!\bigl(\iota_\Pi^{\mathsf T}Y^{-1}\iota_\Pi\bigr)^{-1/(n-l)},
    \]
    where $\iota_\Pi$ is the \(n\times(n-l)\) matrix whose columns span \(\Pi\). By Proposition \ref{prop:4:3}, the usual Busemann-Selberg function \(\mathfrak b_{\alpha_\Phi,Z}\) is \(w\)-invariant:
    \[
    \mathfrak b_{\alpha_\Phi,Z}(Y) = \mathfrak b_{\alpha_\Phi,Z}(w\cdot Y).
    \]
    It remains to check
    \[
        \det(\iota_\Pi^{\mathsf{T}}Y^{-1}\iota_\Pi) = \det(\iota_\Pi^{\mathsf{T}}(w.Y)^{-1}\iota_\Pi).
    \]
    Since \(w\) preserves the boundary component \(\Psi\), its action on \(\iota_\Pi\) satisfies
    \[
    (w^{\mathsf{T}})^{-1}\iota_\Pi = \iota_\Pi w', w'\in GL(n-l,\mathbb{R}).
    \]
    In fact, matrices $W$ and $(w^{\mathsf{T}})^{-1}W$ represent boundary components $\Psi$ and $w^{-1}.\Psi$, which are assumed to be the same component. Therefore,
    Hence
    \[
    \begin{split}
        & \det(\iota_\Pi^{\mathsf{T}}(w.Y)^{-1}\iota_\Pi) = \det(((w^{\mathsf{T}})^{-1}\iota_\Pi)^{\mathsf{T}}Y^{-1}((w^{\mathsf{T}})^{-1}\iota_\Pi)) \\
        & = \det((\iota_\Pi w')^{\mathsf{T}}Y^{-1}(\iota_\Pi w')) = \det(w')^2\det(\iota_\Pi^{\mathsf{T}}Y^{-1}\iota_\Pi).
    \end{split}
    \]
    Finally, Proposition~\ref{prop:4:3} ensures that \(\pi_\Pi(w)\) has finite order, so \(\det(w')^2= 1\). Therefore the two determinants agree, and the \(l\)-th Busemann-Selberg function is \(w\)-invariant.
\end{proof}
\subsection{Riemannian Dihedral Angles in Dirichlet-Selberg Domains}\label{subsec:3:2}
For Dirichlet domains in hyperbolic spaces, a critical property is the independence of Riemannian dihedral angles from base point choices. While this fails for Dirichlet-Selberg domains in $\mathcal{X}_n$, understanding the dependence of this angle on the choice of base point is crucial for the proof of the main theorem.
% In this subsection, we prove a formula for the Riemannian dihedral angles between hyperplanes in $\mathcal{X}_n$.

As we defined earlier, a plane $P\subset \mathcal{X}_n$ of codimension $k$ is a non-empty intersection of $k$ linearly-independent hyperplanes. In addition, each of these hyperplanes is a perpendicular plane for an indefinite matrix $A\in \mathrm{Sym}_n(\mathbb{R})$. Therefore, the plane can be described as
\[
P = \left(\bigcap_{i=1}^{k}A_i^\perp\right) = \mathrm{span}(A_1,\dots,A_k)^\perp.
\]
In a Dirichlet-Selberg domain, a pair of adjacent faces of codimension $k$ spans two planes $P$ and $P'$ that intersect along $P\cap P'$ of codimension $k+1$. They can be described as
\begin{equation}\label{equ:3:1}
    P = \mathrm{span}(A_1,\dots,A_{k-1},B)^\perp,\ P' = \mathrm{span}(A_1,\dots,A_{k-1},B')^\perp,
\end{equation}
for linearly independent indefinite matrices $A_1,\dots, A_{k-1}, B$, and $B'\in \mathrm{Sym}_n(\mathbb{R})$.
\begin{lem}\label{lem:3:2}
    Let $P$ and $P'$ be planes described as in \eqref{equ:3:1}. Then, for any point $X\in P\cap P'$, the Riemannian dihedral angle $\angle_X(P,P')$ is given by:
    \[
    \angle_X(P,P') = \arccos\frac{\left(\bigwedge_{i=1}^{k-1}A_i\wedge B,\bigwedge_{i=1}^{k-1}A_i\wedge B'\right)_{X^{-1}}}{\sqrt{\left\lVert\bigwedge_{i=1}^{k-1}A_i\wedge B\right\rVert_{X^{-1}}\cdot\left\lVert\bigwedge_{i=1}^{k-1}A_i\wedge B'\right\rVert_{X^{-1}}}},
    \]
    where $(\cdot,\cdot)_{X^{-1}}$ denotes the inner product, and $||\cdot||_{X^{-1}}$ the norm, on the exterior algebra $\bigwedge^k(\mathrm{Sym}_n(\mathbb{R}))$ induced by the inner product on $\mathrm{Sym}_n(\mathbb{R})$:
    \[
    \langle A_1,A_2\rangle_{X^{-1}} = \mathrm{tr}(XA_1XA_2),\ \forall A_1,A_2\in \mathrm{Sym}_n(\mathbb{R}).
    \]
\end{lem}
\begin{proof}
    In the hypersurface model, the tangent space $T_XP$ is a subspace of $T_X\mathbb{R}^{n(n+1)/2}$:
    \[
    T_XP = \left\{C\in T_X\mathbb{R}^{n(n+1)/2}\left|\mathrm{tr}(A_iC) = 0,\ \mathrm{tr}(BC) = 0,\ \mathrm{tr}(X^{-1}C) = 0\right.\right\}.
    \]
    Similarly:
    \[
    T_XP' = \left\{C\in T_X\mathbb{R}^{n(n+1)/2}\left|\mathrm{tr}(A_iC) = 0,\ \mathrm{tr}(B'C) = 0,\ \mathrm{tr}(X^{-1}C) = 0\right.\right\}.
    \]
    Recall that the dihedral angles between linear subspaces of $T_X\mathbb{R}^{n(n+1)/2}$ are measured by the inner product given by the Killing form:
    \[
    \langle C,C'\rangle_X = \mathrm{tr}(X^{-1}CX^{-1}C').
    \]
    Thus, the dihedral angle between $T_XP$ and $T_XP'$ is equal to their orthogonal complements with respect to $\langle -,-\rangle_X$. These can be expressed explicitly in terms of bases:
    \[
    \begin{split}
        & (T_XP)^\perp = \mathrm{span}(X,XA_1X,\dots, XA_{k-1}X, XBX), \\
        & (T_XP')^\perp = \mathrm{span}(X,XA_1X,\dots, XA_{k-1}X, XB'X).
    \end{split}
    \]
    The angle between these complementary spaces is then given by
    \[
    \resizebox{\textwidth}{!}{$
     \begin{aligned}
        \arccos\frac{\det\begin{pmatrix}
            \langle XA_iX, XA_jX\rangle_X & \langle XA_iX, XBX\rangle_X & \langle XA_iX, X\rangle_X \\ \langle XB'X, XA_jX\rangle_X & \langle XB'X, XBX\rangle_X & \langle XB'X, X\rangle_X \\ \langle X, XA_jX\rangle_X & \langle X, XBX\rangle_X & \langle X, X\rangle_X \\ 
        \end{pmatrix}_{1\leq i,j\leq k-1}}{\sqrt{\det\begin{pmatrix}
            \langle XA_iX, XA_jX\rangle_X & \langle XA_iX, XBX\rangle_X & \langle XA_iX, X\rangle_X \\ \langle XBX, XA_jX\rangle_X & \langle XBX, XBX\rangle_X & \langle XBX, X\rangle_X \\ \langle X, XA_jX\rangle_X & \langle X, XBX\rangle_X & \langle X, X\rangle_X \\ 
        \end{pmatrix}_{1\leq i,j\leq k-1}\det\begin{pmatrix}
            \langle XA_iX, XA_jX\rangle_X & \langle XA_iX, XB'X\rangle_X & \langle XA_iX, X\rangle_X \\ \langle XB'X, XA_jX\rangle_X & \langle XB'X, XB'X\rangle_X & \langle XB'X, X\rangle_X \\ \langle X, XA_jX\rangle_X & \langle X, XB'X\rangle_X & \langle X, X\rangle_X \\ 
        \end{pmatrix}_{1\leq i,j\leq k-1}}}.
        \end{aligned}
     $}
    \]
    To simplify this expression, note that
    \[
    \langle XA_iX,XA_jX\rangle_X = \mathrm{tr}(X^{-1}XA_iXX^{-1}XA_jX) = \mathrm{tr}(XA_iXA_j) = \langle X_i,X_j\rangle_{X^{-1}}.
    \]
    Additionally, since $X\in P\cap P'$, we have that
    \[
    \langle X,XA_iX\rangle_X = \mathrm{tr}(A_iX) = 0,\ \langle X,XBX\rangle_X = 0,\ \langle X,XB'X\rangle_X = 0,
    \]
    and 
    \[
    \langle X,X\rangle_X = \mathrm{tr}(I_n) = n.
    \]
    These simplify the formula into the form as presented in Lemma \ref{lem:3:2}.
\end{proof}
\begin{exm}
    If $P = B^\perp$ and $P' = B'^\perp$ are hyperplanes, then the Riemannian dihedral angle at any $X\in P\cap P'$ is given as
    \[
    \angle_{X}(P,P') = \arccos\frac{\mathrm{tr}(XBXB')}{\sqrt{\mathrm{tr}((XB)^2)\mathrm{tr}((XB')^2)}}.
    \]
\end{exm}
An essential corollary of Lemma \ref{lem:3:2} is the following asymptotic behavior of Riemannian dihedral angles to the Satake boundary:
\begin{prop}\label{prop:3:3}
    Suppose that $P$ and $P'$ are planes of the same dimension in $\mathcal{X}_n$, and $P\cap P'$ is of codimension $1$ in both $P$ and $P'$. Assume further that $\Pi$ is a Satake plane of type $n-k$ in $\overline{\mathcal{X}_n}$, and is transverse to both $\overline{P}$ and $\overline{P'}$. Then for each $\alpha\in \overline{P}\cap \overline{P'}\cap \Pi$ and $Y\in P\cap P'$, the limit of Riemannian dihedral angle
    \[
    \lim_{\epsilon\to 0_+}\angle_{\alpha+\epsilon Y}(P,P') = \angle_{\pi(\alpha)}(\pi(\overline{P}\cap \Pi),\pi(\overline{P'}\cap \Pi)).
    \]
    Here, $\pi$ is the diffeomorphism from $\Pi$ to $\mathcal{X}_{n-k}$ given in Definition \ref{def:3:5}.
\end{prop}
\begin{proof}
    Without loss of generality, let $\operatorname{Col}(\Pi) = \mathrm{span}(\mathbf{e}_{1},\dots,\mathbf{e}_{n-k})$, and let
    \[
    P = \mathrm{span}(A_1,\dots,A_{l-1},B)^\perp,\ P' = \mathrm{span}(A_1,\dots,A_{l-1},B')^\perp,
    \]
    For $i=1,\dots, l-1$, denote the minors of the first $(n-k)$ rows and columns of $A_i$, $B$ and $B'$ by $A_{i,0}$, $B_0$, and $B'_0$, respectively. Then,
    \[
    \pi(\overline{P}\cap \Pi) = \mathrm{span}(A_{1,0},\dots,A_{l-1,0},B_0)^\perp,\ \pi(\overline{P'}\cap \Pi) = \mathrm{span}(A_{1,0},\dots,A_{l-1,0},B'_0)^\perp.
    \]
    The transversality of $\Pi$ to $\overline{P}$ and $\overline{P'}$ ensures that $A_{0,1}$, \dots, $A_{0,l-1}$, $B_0$ and $B_0'$ are linearly independent. By Lemma \ref{lem:3:2}, we have
    \[
    \angle_{\pi(\alpha)}(\pi(\overline{P}\cap \Pi),\pi(\overline{P'}\cap \Pi)) = \arccos\frac{\left(\bigwedge_{i=1}^{l-1}A_{i,0}\wedge B_0,\bigwedge_{i=1}^{l-1}A_{i,0}\wedge B'_0\right)_{\alpha_0^{-1}}}{\sqrt{\left|\left|\bigwedge_{i=1}^{l-1}A_{i,0}\wedge B_0\right|\right|_{\alpha_0^{-1}}\cdot\left|\left|\bigwedge_{i=1}^{l-1}A_{i,0}\wedge B'_0\right|\right|_{\alpha_0^{-1}}}},
    \]
    where $\alpha = \mathrm{diag}(\alpha_0,O)$, i.e., $\alpha_0 = \pi(\alpha)$ is the minor consisting of the first $(n-k)$ rows and columns of $\alpha$. This suggests that $\alpha^{1/2}A_i\alpha^{1/2} = \mathrm{diag}(\alpha_0^{1/2}A_{i,0}\alpha_0^{1/2},O)$, for $i = 1,\dots,l-1$. Hence, as $\epsilon\to 0$, the inner products for the Riemannian angle have the following limits:
    \[
    \begin{split}
        & \lim_{\epsilon\to 0}\langle A_i,A_j\rangle_{(\alpha+\epsilon Y)^{-1}} = \mathrm{tr}(\alpha A_i\alpha A_j) \\
        & = \mathrm{tr}((\alpha^{1/2}A_i\alpha^{1/2})(\alpha^{1/2}A_j\alpha^{1/2})) = \mathrm{tr}((\alpha_0^{1/2}A_{i,0}\alpha_0^{1/2})(\alpha_0^{1/2}A_{j,0}\alpha_0^{1/2}))\\
        & = \mathrm{tr}(\alpha_0A_{i,0}\alpha_0A_{j,0}) = \langle A_{i,0},A_{j,0}\rangle_{\alpha_0^{-1}}.
    \end{split}
    \]
    By substituting these limits into the expression of $\angle_{\alpha+\epsilon Y}(P,P')$, we obtain that
    \[
    \lim_{\epsilon\to 0_+}\angle_{\alpha+\epsilon Y}(P,P') = \angle_{\pi(\alpha)}(\pi(\overline{P}\cap \Pi),\pi(\overline{P'}\cap \Pi)).
    \]    
\end{proof}
\begin{exm}
    Given hyperplanes $A^\perp$ and $B^\perp$ in $\mathcal{X}_3$, $A = \mathrm{diag}(A_0,0)$ and $B = \mathrm{diag}(B_0,0)$, where
    \[
    A_0 = \left(\begin{array}{cc}
        0 & -1\\
        -1 & 1\\
    \end{array}\right),\ B = \left(\begin{array}{cc}
        1 & -1\\
        -1 & 0\\
    \end{array}\right)
    \]
    Then, $A_0^\perp = \overline{A^\perp}\cap \partial_S(\mathbf{e}_1,\mathbf{e}_2)$ and $B_0^\perp = \overline{B^\perp}\cap \partial_S(\mathbf{e}_1,\mathbf{e}_2)$ are identified with geodesics in $\mathbf{H}^2$, meeting at the point
    \[
    \alpha_0 = \left(\begin{array}{cc}
        1 & 1/2 \\
        1/2 & 1
    \end{array}\right),
    \]
    with a Riemannian angle of $2\pi/3$. By Proposition \ref{prop:3:3}, for any line in $A^\perp\cap B^\perp$ that diverges to $\alpha = \mathrm{diag}(\alpha_0,0)\in\partial_S\mathcal{X}_3$, the Riemannian dihedral angle between $A^\perp$ and $B^\perp$ based at a point on this line will converge to $2\pi/3$, when the base point diverges to $\alpha$.
\end{exm}
\vspace{12pt}
\section{Proof of the Main Theorem}\label{Sec:5}
Let $D = DS(X,\Gamma)\subset \mathcal{X}_3$ be an exact, finitely-sided Dirichlet-Selberg domain of finite volume satisfying the tiling condition. Recally that $D$ has up to finitely many Satake boundary components of type two, and these components meet only at certain Satake vertices of type one. We prove Theorem \ref{thm:1:2} (the main result of this paper) in two steps:

\begin{enumerate}
    \item In Subsection~\ref{subsec:5:1}, we construct a subset
    \[
    D^{(1)}\subset D,
    \]
    namely a disjoint union of small neighborhoods around each Satake vertex of type one, such that $D^{(1)}$ meets only the faces incident to those vertices. We then show there exists \(r_1>0\) so that for every \(X\in D^{(1)}\), the \(r_1\)-ball centered at its image \(\widetilde{X}\in M:=D/\!\sim\) is complete. Remove these neighborhoods from $M$, we obtain a manifold (or orbifold) with boundary, denoted by $M'$. Let $D'\subset D$ be the preimage of $M'$. By construction, the Satake boundary components of type two in $D'$ are now pairwise disjoint.
    \item In Subsection \ref{subsec:5:2}, we similarly define
    \[
    D^{(2)}\subset D'
    \]
    as a disjoint union of neighborhoods around each remaining boundary component of type two, meeting only their incident faces. We then prove there exists \(r_2>0\) so that for all \(X\in D^{(2)}\subset D'\), the \(r_2\)-ball around its image in \(M'\) is complete. Since the complement 
    \(\,D\setminus (D^{(1)}\cup D^{(2)})\)\, is bounded, it follows that \(M=D/\!\sim\) is complete.
\end{enumerate}
Throughout the proof we assume, without loss of generality, that the domain \(D\) is centered at $X = I$, the identity matrix.
\subsection{Part I: Behavior Near Satake Vertices of Type One}\label{subsec:5:1}
We begin by analyzing the cycle structure of Satake vertices of type $1$. Since Busemann-Selberg functions depend on chosen reference points, we select them so as to satisfy a natural \textbf{vertex-cycle condition}:
\begin{lem}\label{lem:5:1}
    Let $\alpha\in \partial_S D$ be a Satake vertex of type $1$, and let $\Phi$ be a Satake face of type $2$ containing $\alpha$. Denot by $\eta$ and $\eta'$ the two edges of $\Phi$ meet at $\alpha$, and $w$ be any word in the cycle of edges sending $\eta$ to $\eta'$. Then $w$ also fixes $\alpha$. Writing the boundary component $\Pi = \mathrm{span}(\Phi)$, there exists a constant $C\geq 1$, depending only on $D$ and $\alpha$, such that for all $Y\in\mathcal{X}_3$,\
    \[
        C^{-1}\mathfrak{b}^{(1)}_{\Pi;\alpha,X}(Y)\leq \mathfrak{b}^{(1)}_{\Pi;\alpha,X}(w.Y)\leq C\mathfrak{b}^{(1)}_{\Pi;\alpha,X}(Y).
    \]
\end{lem}
\begin{proof}
    First, let $w_0$ be any cycle of the edge $\eta$. By Proposition \ref{prop:4:3}, the restriction of $w_0$ to $\mathrm{span}(\Phi)$ has finite order, hence is not loxodromic. It follows that $w_0$ fixes every point of $\Phi$ and preserves the Busemann-Selberg function $\mathfrak{b}^{(1)}_{\Pi;\alpha,X}$.
    
    Next, suppose $w$ and $w'$ are two words in the edge cycle sending $\eta$ to $\eta'$. Both preserve the boundary component $\Pi$, so both scale $\mathfrak{b}^{(1)}_{\Pi;\alpha,X}$ by the same factor $C_\Phi>0$. Reversing the cycle rescales by $C_\Phi^{-1}$. That is,
    \[
    \mathfrak{b}^{(1)}_{\Pi;\alpha,X}(w. Y)
    = C_\Phi\,\mathfrak{b}^{(1)}_{\Pi;\alpha,X}(Y),
    \quad
    \mathfrak{b}^{(1)}_{\Pi;\alpha,X}(w^{-1}. Y)
    = C_\Phi^{-1}\,\mathfrak{b}^{(1)}_{\Pi;\alpha,X}(Y).
    \]
    Since there are only finitely many such vertices $\alpha_i$ in the orbit of $\alpha$ and faces $\Phi_i$ through each $\alpha_i$, we may set
    \[
    C = \max_{\alpha_i,\Phi_i}\{C_{\Phi_i},C_{\Phi_i}^{-1}\},
    \]
    which yields the desired uniform bound.
\end{proof}
By Proposition \ref{prop:4:3} and Lemma \ref{lem:5:1}, we may now choose reference-point-free Busemann-Selberg functions $\mathfrak{b}_{\alpha_i}$ and $\mathfrak{b}^{(1)}_{\Pi_i;\alpha_i}$ for each $\alpha_i\in[\alpha]$, so that:
\begin{itemize}
    \item If $\alpha_j = w.\alpha_i$ for some $w$ in the vertex cycle, then
    \[
    \mathfrak{b}_{\alpha_i}(Y) = \mathfrak{b}_{\alpha_j}(w.Y),\ \text{for any}\ Y\in\mathcal{X}_3.
    \]
    \item If $\eta_i$, $\eta_j$ are edges in the same edge orbit $[\eta]$, with $\eta_j = w.\eta_i$ and corresponding boundary components $\Pi_i$, $\Pi_j$, then
    \[
    C^{-1}\mathfrak{b}^{(1)}_{\Pi_i;\alpha_i}(Y)\leq \mathfrak{b}^{(1)}_{\Pi_j;\alpha_j}(w.Y)\leq C\mathfrak{b}^{(1)}_{\Pi_i;\alpha_i}(Y).
    \]
\end{itemize}
We denote the associated horoballs (independent of reference points) by $B(\alpha_i,r)$ and $B^{(1)}_{\Pi_i}(\alpha_i,r)$. Since $D$ has only finitely many faces, we define a neighborhood of $\alpha$ inside $D$ by
\[
B_D^{(1)}(\alpha,r) = \bigcap_{{\Phi\ni \alpha}}B_{\mathrm{span}(\Phi)}^{(1)}(\alpha,r),
\]
where the intersection runs over all type-$2$ Satake faces $\Phi\ni\alpha$.

As the parameter $r$ approaches to zero, the lemma below implies that the neighborhood $B_D^{(1)}(\alpha,r)$ shrinks to the Satake vertex $\alpha$:
\begin{lem}\label{lem:5:2}
    For any $r>0$, the closure $\overline{B_D^{(1)}(\alpha,r)}\cap \overline{D}$ contains a neighborhood of $\alpha$ within $\overline{D}$. Moreover, the intersection
    \[
    \bigcap_{m=1}^\infty \left(\overline{B_D^{(1)}(\alpha,1/m)}\cap \overline{D}\right) = \{\alpha\}.
    \]
\end{lem}
\begin{proof}
    For the first assertion, we need to show that $\overline{B_{\Pi}^{(1)}(\alpha,r)}$ contains a neighborhood of $\alpha$ in $\overline{D}$, where $\Pi = \mathrm{span}(\Phi)$ and $\Phi$ is any type $2$ Satake face containing $\alpha$.

    To establish this, let $S$ be a sphere in $\mathbb{R}\mathbf{P}^5$ centered at $\alpha$ that intersects every face or Satake face of $\overline{D}$ containing $\alpha$. Then, the convex hull of $\alpha\sqcup \left(S\cap \overline{D}\right)$ contains a neighborhood of $\alpha$ in $\overline{D}$. We aim to show that this neighborhood is contained in $\overline{B_{\Pi}^{(1)}(\alpha,r)}$ when the radius of $S$ is sufficiently small. This is justified by showing that the line segment from $\alpha$ to $\alpha+\epsilon X$ is entirely contained within $\overline{B_{\Pi}^{(1)}(\alpha,r)}$, where $X$ is a point in $S\cap \overline{D}$, and $\epsilon>0$ depends on $X$. Such points $X$ can be categorized into three cases:
    \begin{enumerate}[(i)]
        \item $X\in D$,
        \item $X\in \Phi$, or
        \item $X$ lies on a type $2$ Satake face distinct from $\Phi$.
    \end{enumerate}
    
    Case (i): When $X\in D$, this containment is straightforward.

    Case (ii): When $X\in \Phi$, Lemma \ref{lem:3:6} implies that for any smooth curve $\alpha+\epsilon X +t Y$ approaching $\alpha+\epsilon X$ in $\overline{D}$, where $Y\in\mathcal{X}_3$, the Busemann-Selberg function $\mathfrak{b}^{(1)}_{\Pi;\alpha}(\alpha+\epsilon X +t Y)$ converges to $\mathfrak{b}_{\pi(\alpha)}(\pi(\alpha+\epsilon X))$, a value less than $r$ for sufficiently small $\epsilon>0$. Proposition \ref{prop:4:2} then implies that $\alpha+\epsilon X$ is on the type-one horosphere $\Sigma_{\Pi}^{(1)}(\alpha,r)$. Thus, the segment from $\alpha$ to $\alpha+\epsilon X$ remains within $\overline{B_{\Pi}^{(1)}(\alpha,r)}$. 
    
    Case (iii): When $X$ is in a type $2$ Satake face distinct from $\Phi$, Lemma \ref{lem:3:5} ensures that the entire line segment from $X$ to $\alpha$ lies within $\overline{B_{\Pi}^{(1)}(\alpha,r)}$. 
    
    Since $S\cap \overline{D}$ is compact and $\mathfrak{b}_{\Pi;\alpha}^{(1)}$ extends continuously to Satake facets in $\partial_S D$ that contain $\alpha$, we can select $\epsilon$ uniformly over all $X\in S\cap \overline{D}$. Thus, a neighborhood of $\alpha$ is indeed contained in $\overline{B_{\Pi}^{(1)}(\alpha,r)}$.

    For the second assertion, notice that for any $\Phi\ni\alpha$ and $\Pi = \mathrm{span}(\Phi)$, the intersection
    \[
    \bigcap_{m=1}^\infty\left(\overline{B_{\Pi}^{(1)}(\alpha,1/m)}\cap \overline{D}\right)
    \]
    excludes all points in $D$; by Lemma \ref{lem:3:6}, it also excludes all points in the Satake face $\Phi$, except for $\alpha$ itself. Taking the intersection over all type $2$ Satake faces $\Phi$ containing $\alpha$ yields:
    \[
    \bigcap_{m=1}^\infty \left(\overline{B_D^{(1)}(\alpha,1/m)}\cap \overline{D}\right) = \bigcap_\Phi\bigcap_{m=1}^\infty\left(\overline{B_{\Pi}^{(1)}(\alpha,1/m)}\cap \overline{D}\right) = \{\alpha\}.
    \]
\end{proof}
Lemma \ref{lem:5:2} ensures the existence of a constant $r>0$ such that the sets $\overline{B_D^{(1)}(\alpha,r)}$ for all type $1$ Satake vertices $\alpha\in\mathcal{F}_S(D)$ form a disjoint union
\[
    \bigsqcup_\alpha B_D^{(1)}(\alpha,r),
\]
consisting of neighborhoods of those type $1$ Satake vertices in $\overline{D}$. The second assertion of the lemma further implies that $r$ can be selected such that each of these component is separated from any face not incident with the corresponding Satake vertex.

We still need a lemma concerning certain relationships between type-one horoballs and classic horoballs based at the same Satake vertex:
\begin{lem}\label{lem:5:3}
    There exists certain constants $r'>0$ and $\epsilon>\epsilon'>0$, such that:
    \begin{enumerate}
        \item For each type-$2$ Satake face $\Phi\ni\alpha$ with $\Pi = \mathrm{span}(\Phi)$, and for any face $G\in \mathcal{F}(D)$ either disjoint from $\Pi$ or precisely incident with $(\alpha,\Pi)$, the set 
        \[
        B(\alpha,r')\backslash B_{\Pi}^{(1)}(\alpha,C^{-1}e^{-2\epsilon}r)
        \]
        lies at distance at least $\epsilon$ from $G$.
        \item If $\eta$ and $\eta'$ are the two edges of $\Phi$ meeting at $\alpha$, and $F,F'\in \mathcal{F}(D)$ are faces precisely incident with $\eta$ and $\eta'$ respectively, then their intersections with
        \[
            B(\alpha,r')\backslash B_{\Pi}^{(1)}(\alpha,C^{-1}e^{-2\epsilon}r),
        \]
        are separated by distance at least $\epsilon'$.
        \item For any two distinct type-$2$ Satake faces $\Phi, \Phi'\ni \alpha$,
        \[
        D\cap B(\alpha,r')\subset D\cap \left(B_\Phi^{(1)}(\alpha,C^{-1}e^{-2\epsilon}r)\cup B_{\Phi'}^{(1)}(\alpha,C^{-1}e^{-2\epsilon}r)\right).
        \]
    \end{enumerate}
\end{lem}
\begin{proof}
    \textbf{(1).} Consider the nested intersections
    \[
    \overline{D}\cap\left(\bigcap_{m=1}^\infty \overline{B(\alpha,1/m)\backslash B_{\Pi}^{(1)}(\alpha,C^{-1}r)}\right).
    \]
    Similar to the proof of Lemma \ref{lem:5:2}, this is the complement of a horoball in $\Phi$ based at $\alpha$, so it is disjoint from any face $G\in \mathcal{F}(D)$ either disjoint from $\Pi$ or incident only at $(\alpha,\Pi)$. By the $1$-Lipschitz property for Busemann-Selberg functions (Proposition \ref{prop:3:4}), for sufficiently small $r'>0$ we obtain a uniform buffer of $3\epsilon$ between
    \[
    B(\alpha,r')\backslash B_{\Pi}^{(1)}(\alpha,C^{-1}r)
    \]
    for all such faces $G\in \mathcal{F}(D)$. This yields our first assertion.
    
    \textbf{(2).} Similarly, for each of the two edges \(\eta,\eta'\) through \(\alpha\), the infinite intersections
    \[
    \overline{F}\cap\left(\bigcap_{m=1}^\infty \overline{B(\alpha,1/m)\backslash B_{\Pi}^{(1)}(\alpha,C^{-1}e^{-2\epsilon}r)}\right)
    \]
    and
    \[
    \overline{F'}\cap\left(\bigcap_{m=1}^\infty \overline{B(\alpha,1/m)\backslash B_{\Pi}^{(1)}(\alpha,C^{-1}e^{-2\epsilon}r)}\right)
    \]
    are the complements of a horoball in the Satake edges $\eta$ and $\eta'$. Hence, by shrinking $r'$ if necessary, one finds $\epsilon'$ so that the corresponding truncated regions are at least \(\epsilon'\) apart, proving, proving our second assertion.
        
    \textbf{(3).} Finally, the infinite intersection
    \[
        \overline{D}\cap\left(\bigcap_{m=1}^\infty \overline{B(\alpha,1/m)}\right)
    \]
    is the union of all Satake faces containing $\alpha$. Since every such face is contained in at least one of the two horoballs $B_\Phi^{(1)}(\alpha,C^{-1}e^{-2\epsilon}r)$ or $B_{\Phi'}^{(1)}(\alpha,C^{-1}e^{-2\epsilon}r)$, it follows that for sufficiently $r'>0$, $D\cap B(\alpha,r')$ is contained in the union of these two type-one horoballs.
\end{proof}
With constants $C$, $r$, $r'$, and $\epsilon$ depending only on the Dirichlet-Selberg domain $D$ defined from Lemmas \ref{lem:5:1} to \ref{lem:5:3}, we are ready to define the set claimed at the beginning: 
    \[
    D^{(1)} = \bigcup_\alpha\left(B_D^{(1)}(\alpha,e^{-2\epsilon}C^{-1}r)\cap B(\alpha,e^{-\epsilon}r')\right).
    \]
As the first half of the proof of the main theorem, we will establish the uniform compactness for balls centered in $D^{(1)}/\sim$.
\begin{proof}[Proof of Theorem \ref{thm:1:2}, first half]
    We aim to prove that for every $\widetilde{X}\in D^{(1)}/\sim$, represented by the point
        \[
        X\in \bigcup_\alpha\left(B_D^{(1)}(\alpha,e^{-2\epsilon}C^{-1}r)\cap B(\alpha,e^{-\epsilon}r')\right),
        \]
    the ball $N(\widetilde{X},\epsilon'/2)$ is compact. Specifically, we will show that for each such $\widetilde{X}$, the preimage of $N(\widetilde{X},\epsilon'/2)$ is contained in the compact region
        \[
        \bigcup_\alpha\left(B_D^{(1)}(\alpha,r)\cap B(\alpha,r')\backslash B(\alpha,e^{-\epsilon'}\mathfrak{b}_{\alpha}(X))\right).
        \]
    Assume, by way of contradiction, that there exists a (piecewise smooth) curve $\gamma$ in $D/\sim$ of length $\leq \epsilon'/2$, connecting $\widetilde{X}$ and another point $\widetilde{Y}$, where $\widetilde{Y}$ is represented by $Y\in\mathcal{X}_n$, and
        \[
        Y\notin \bigsqcup_\alpha\left(B_D^{(1)}(\alpha,r)\cap B(\alpha,r')\right),
        \]
    the disjointness is shown in Lemma \ref{lem:5:2}. Up to a sufficiently small perturbation, we further assume that the preimage of the curve $\gamma$ does not meet any faces of codimension $2$ or more, possibly except for the endpoints $X$ and $Y$. Therefore, the preimage is contained in a disjoint union of certain neighborhoods of Satake vertices $\alpha_1$,\dots, $\alpha_N$, consisting of a collection of segments glued together by the quotient map. For any point $\widetilde{X_i}\in D/\sim$ where two pieces of the preimage are glued together, its preimage consists of two points $X_i\sim X_i'$, paired by a certain facet-pairing transformation $g_i$, in neighborhoods of certain Satake vertices $\alpha_{k_i}$ and $\alpha_{k_{i-1}}$ of type $1$, respectively. We call $X_i$ and $X_i'$ a pair of glued points in $\gamma$.
    
    Consider the first intersection point of $\gamma$ with the set
    \[
    \partial\bigcup_\alpha\left(B_D^{(1)}(\alpha,r)\cap B(\alpha,r')\right),
    \]
    which we denote by $\widetilde{Z}$, represented by $Z\in D$. The preimage of the curve connecting $\widetilde{X}$ and $\widetilde{Z}$ consists of segments $(X_0,X_1')$, $(X_1,X_2')$,\dots, $(X_{m-1},X_m')$, where $X_i\sim X_i'$ are pairs of glued points, and $X = X_0$, $Z = X_m'$ for convenience. We analyze two cases for this intersection point:
        \begin{itemize}
            \item The point $Z$ lies on $\partial B(\alpha',r')$ for a certain Satake vertex $\alpha'$ of type $1$.
            \item The point $Z$ lies on $\partial B_{\Pi'}^{(1)}(\alpha',r)$ for a certain Satake vertex $\alpha'$ of type $1$ and a boundary component $\Pi' = \mathrm{span}(\Phi')$, where $\Phi'$ is a Satake face of type $2$ containing $\alpha'$.
        \end{itemize}
    Assume that the first case occurs. Lemma \ref{lem:5:2} implies that the preimage of the curve restricted to $B_D^{(1)}(\alpha,r)\cap B(\alpha,r')$ does not intersect any face not meeting $\alpha$. Therefore, for each pair of glued points $X_i\sim X_i'$ in the curve connecting $\widetilde{X}$ and $\widetilde{Z}$, Proposition \ref{prop:4:3} implies the equality 
    \[
    \mathfrak{b}_{\alpha_{k_i}}(X_i) = \mathfrak{b}_{\alpha_{k_{i-1}}}(X_i').
    \]
    Combining this with the Lipschitz condition for Busemann-Selberg functions (Proposition \ref{prop:3:4}), we deduce that
        \[
        \mathfrak{b}_{\alpha'}(Z)<e^{\epsilon'}\mathfrak{b}_\alpha(X)<r',
        \]
    given that the segments in the preimage of the curve connecting $\widetilde{X}$ and $\widetilde{Z}$ have a total length less than $\epsilon'$. However, this contradicts the assumption $Z\in \partial B(\alpha',r')$.

    Now assume that the second case occurs. Let $(\Pi',\alpha') = (\Pi_{k_{m-1}},\alpha_{k_{m-1}})$, and inductively define that $(\Pi_{k_{i-1}},\alpha_{k_{i-1}})$ to be the pair of boundary component with Satake vertex taken to $(\Pi_{k_i},\alpha_{k_i})$ by $g_i$. Then $\alpha_{k_0} = \alpha$, and $\Pi_{k_0}$ is one of the boundary components containing $\alpha$. Denote it by $\Pi$, the assumption implies
        \[
        \mathfrak{b}^{(1)}_{\Pi,\alpha}(X)\leq e^{-2\epsilon}C^{-1}r,\ \mathfrak{b}^{(1)}_{\Pi',\alpha'}(X') = r.
        \]
    Let $\Phi_{k_i}$ be the Satake face contained in $\Pi_{k_i}$. Since $X_i$ and $X_i'$ lie in the interior of facets of $D$, $g_i.\Phi_{k_{i-1}}$ and $\Phi_{k_i}$ share at least a side. According to the choice of type-one Busemann-Selberg functions, their values $\mathfrak{b}^{(1)}_{\Pi_{k_{i-1}},\alpha_{k_{i-1}}}(X_i')$ and $\mathfrak{b}^{(1)}_{\Pi_{k_i},\alpha_{k_i}}(X_i)$ differs by a constant multiplier $\leq C$. Combining this fact with the $1$-Lipschitz condition for type-one Busemann-Selberg functions (Proposition \ref{prop:3:4}), there is a certain $X_j$ such that
        \[
        e^{-2\epsilon}C^{-1}r\leq \mathfrak{b}^{(1)}_{\Pi_{k_j},\alpha_{k_j}}(X_j)\leq e^{-\epsilon}r.
        \]
    The third assertion in Lemma \ref{lem:5:3} implies that for each $\alpha$, the union
    \[
    \bigsqcup_{\Pi\ni\alpha}B(\alpha,r')\backslash B^{(1)}_\Pi(\alpha,C^{-1}e^{-2\epsilon}r)
    \]
    is disjoint. The first assertion in Lemma \ref{lem:5:3} implies that the preimage of the curve from $\widetilde{X_j}$ to $\widetilde{Z}$ restricted to the component for $\Pi$ of the union above does not meet faces not incident with the two edges $\eta$ and $\eta'$ in $\Pi$. Moreover, the second assertion in Lemma \ref{lem:5:3} implies that balls centered at points in the cycle of $X_j$ with radius $\epsilon'/2$ are disjoint and do not intersect facets that precisely incident with a different Satake line. Therefore, along the preimage of the curve from $\widetilde{X}_j$ to $\widetilde{Z}$, the corresponding facet-pairing transformations compose into a word $w$, which maps $\Pi_{k_{j-1}}$ to $\Pi_{k_m}$, ensuring that $w.\Phi_{k_{j-1}}$ and $\Phi_{k_m}$ share at least a side. Consequently, the values $\mathfrak{b}^{(1)}_{\Pi',\alpha'}(Z)$ is strictly less than $r$, contradicting the assumption that $Z$ lies on $\partial B_{\Pi'}^{(1)}(\alpha',r)$.

    This completes the proof of the first half of Theorem \ref{thm:1:2}.
\end{proof}
\begin{rmk}
    We can refine the construction by considering smaller neighborhoods of these Satake vertices, still denoted by $D^{(1)}$, such that any points $X,X'\in\partial D$ paired by a side pairing transformation are either both included in or excluded from $D^{(1)}$.
\end{rmk}
\subsection{Part II: Behavior Near Satake Faces of Type Two}\label{subsec:5:2}
We have derived a polytope $D' = D\backslash D^{(1)}$ with unpaired boundary components that does not contain Satake vertices of type $1$, and contains only disjoint Satake faces of type $2$. In this subsection, we proceed to analyze the cycles of these type $2$ Satake faces.

The first lemma in this subsection is parallel to Lemma \ref{lem:5:2} and proved similarly:
\begin{lem}\label{lem:5:4}
    Let $\Phi_i$ be a Satake face of $D$, and $\alpha_{\Phi_i}$ be an interior point of $\mathrm{span}(\Phi_i)$. Then, for any $r>0$, the closure $\overline{B(\alpha_{\Phi_i},r)}\cap \overline{D'}$ contains a neighborhood of $\Phi_i$ in $\overline{D'}$. Furthermore,
    \[
    \bigcap_{m=1}^\infty\left(\overline{B(\alpha_{\Phi_i},1/m)}\cap \overline{D'}\right) = \Phi_i\cap \overline{D'}.
    \]
\end{lem}
If the Satake face $\Phi_i$ is $2$-dimensional, the proof requires us to decompose the set $\overline{B(\alpha_{\Phi_i},r)}\cap \overline{D'}$ into three mutually exclusive parts:
\begin{itemize}
    \item Points contained in the $\delta$-neighborhood of a face precisely incident with a \textbf{vertex} of $\Phi_i$ at $\Pi_i$,
    \item Points not of the previous type, while contained in the $\epsilon$-neighborhood of a face precisely incident with an \textbf{edge} of $\Phi_i$ at $\Pi_i$, and
    \item All other points in $\overline{B(\alpha_{\Phi_i},r)}\cap \overline{D'}$.
\end{itemize}
As shown in the following lemmas, we can choose certain constants $\epsilon,\delta>0$ such that the second part is a disjoint union corresponding to the edges of $\Phi_i$.
\begin{lem}\label{lem:5:5}
    Let $P_1$ and $P_2$ be hyperplanes in $\mathcal{X}_3$ passing through $I$, and let the Riemannian dihedral angle satisfy
    \[
    0<\theta_1\leq \angle_I(P_1,P_2)\leq \theta_2<\pi.
    \]
    Then for each $\delta>0$, there exists $\epsilon>0$ depending on $\delta$, $\theta_1$ and $\theta_2$, such that
    \[
    N(I,1)\cap N(P_1,\epsilon)\cap N(P_2,\epsilon)\subset N(I,1)\cap N(P_1\cap P_2,\delta).
    \]
    Here, $N(P,r)$ denotes the $r$-neighborhood of $P$ in $\mathcal{X}_3$.
\end{lem}
\begin{proof}
    Consider the space of all pairs of hyperplanes in $\mathcal{X}_3$ passing through $I$ with topology induced by their normal vectors. There exists a value $\epsilon$ satisfying the inclusion condition, depending on the hyperplane pair $(P_1,P_2)$. 
    
    This defines a function on the space of hyperplane pairs, which is continuous and is strictly positive whenever the dihedral angle $\angle_I(P_1,P_2)$ is bounded away from $0$ and $\pi$. Since the space of hyperplane pairs is compact, there exists a constant $\epsilon>0$ such that the inclusion condition holds for all such pairs $(P_1,P_2)$. 
\end{proof}
\begin{lem}\label{lem:5:6}
    Let $\eta$ and $\eta'$ be adjacent edges of the Satake face $\Phi$, such that $\eta\cap \eta' = \alpha$. Let $F$ and $F'$ be faces of $D$ precisely incident with $\eta$ and $\eta'$, respectively. Then there is a certain $r>0$, such that for every sufficiently small $\delta>0$, there is a certain $\epsilon>0$, satisfying
    \[
        \left(B(\alpha_{\Phi},r)\cap \left(F\backslash N(G,\delta)\right)\right)\cap\left(B(\alpha_{\Phi},r)\cap \left(F'\backslash N(G,\delta)\right)\right) = \varnothing,
    \]
    and are separated from each other by distance at least $\epsilon$. Here, $G = F\cap F'$ if it is non-empty. If $F\cap F' = \varnothing$, $G$ is an arbitrary face that is precisely incident with $\alpha$. 
\end{lem}
\begin{proof}
    \textbf{Case (1)}. If $F\cap F' = \varnothing$, we have $\overline{F}\cap \overline{F'}\cap \overline{B(\alpha_{\Phi},r)} = \alpha$. For any $G$ precisely incident with $\alpha$, Lemma \ref{lem:3:3} implies that the completion $\overline{N(G,\delta)}$ contains a neighborhood of $\alpha$ in $\overline{D}$. Therefore, 
    \[
    \overline{F\backslash N(G,\delta)}\ \text{and}\ \overline{F'\backslash N(G,\delta)}
    \]
    does not meet in $\overline{B(\alpha_{\Phi},r)}$, making them of a positive distance away from each other.
    
    \textbf{Case (2)}. Suppose $F\cap F'$ is a face of $D$ precisely incident with $\alpha$ at $\Pi = \mathrm{span}(\Phi)$. Without loss of generality, consider the case when $F$ and $F'$ are facets. According to Proposition \ref{prop:3:3}, the angle $\angle_X(F,F')$ satisfies
    \[
    \angle_X(F,F')\to \angle_\alpha(\eta,\eta') := \theta\in (0,\pi),
    \]
    as the base point $X\in F\cap F'$ is asymptotic to $\alpha$. By Lemma \ref{lem:5:4}, there exists $r>0$ such that
    \[
    \frac{\theta}{2}\leq \angle_X(F,F')\leq \frac{\theta+\pi}{2},
    \]
    for all $X\in F\cap F'\cap B(\alpha_{\Phi},r)$.

    Now fix $X\in F\cap F'\cap B(\alpha_{\Phi},r)$. There exists $g\in SL(3,\mathbb{R})$ such that $g.X = I$. Moreover,
    \[
    \angle_I(g.F,g.F') = \angle_X(F,F')\in \left[\frac{\theta}{2}, \frac{\theta+\pi}{2}\right],
    \]
    where $\mathrm{span}(g.F)$ and $\mathrm{span}(g.F')$ are hyperplanes in $\mathcal{X}_3$ passing through $I$. By Lemma \ref{lem:5:5}, there exists $\epsilon>0$ such that
    \[
    N(I,1)\cap N(g.F,\epsilon)\cap N(g.F',\epsilon)\subset N(I,1)\cap N(g.F\cap g.F',\delta).
    \]
    Pulling back by $g^{-1}$:
    \[
    N(X,1)\cap N(F,\epsilon)\cap N(F',\epsilon)\subset N(X,1)\cap N(F\cap F',\delta).
    \]
    Since the number $\epsilon>0$ is independent of $X$, we apply this for all points $X$ in $X\in F\cap F'\cap B(\alpha_{\Phi},r)$ and deduce
    \[
    B(\alpha_{\Phi},r)\cap N(F\cap F',1)\cap N(F,\epsilon)\cap N(F',\epsilon)\subset N(F\cap F',\delta).
    \]
    We claim that for any $Y\in N(F,\epsilon)$ and $Y'\in N(F',\epsilon)$ outside of $N(F\cap F',1)$, the distance $d(Y,Y')\geq 2\epsilon$ as well. Assume this is not true, then for $X\in F\cap F'$ and lines $s$ and $s'$: $[0,1]\to\mathcal{X}_3$ from $X$ to $Y$ and $Y'$, the distance from $s(t)$ to $s'$ strictly increases as $t$ increases from $0$ to $1$. However, when $s(t)$ lies in $N(F\cap F',1)\backslash N(F\cap F',\delta+\epsilon)$, its distance to $s'$ is at least $2\epsilon$, contradicting the assumption $d(Y,Y')<2\epsilon$.

    Thus, we may eliminate $N(F\cap F',1)$ from the inclusion above, yielding that
    \[
    B(\alpha_{\Phi},r)\cap (F\backslash N(F\cap F',\delta)) \text{ and } B(\alpha_{\Phi},r)\cap (F'\backslash N(F\cap F',\delta))
    \]
    are separated by distance at least $\epsilon$.
\end{proof}
For each two-dimensional Satake face $\Phi$ and one-dimensional Satake edge $\eta$ in $D'$, Proposition \ref{prop:4:3} provides fixed points $\alpha_{\Phi}$ and $\alpha_{\eta}$ under the corresponding Satake cycles. Moreover, we may choose these points and their Busemann-Selberg functions so that, whenever $\Phi_j = w.\Phi_i$ for some word $w$ in the cycle $[\Phi]$, one has
\[
\alpha_{\Phi_j} = w.\alpha_{\Phi_i}\text{ and }\mathfrak{b}_{\alpha_{\Phi_i}}(Y) = \mathfrak{b}_{\alpha_{\Phi_j}}(w.Y),\quad \forall Y\in\mathcal{X}_3,
\]
and similarly for Satake edge and vertex cycles in $D'$. We will show that there is $r>0$ so that
\[
    D^{(2)} = \bigcup_{\Pi} B(\alpha_{\Phi_\Pi},r)
\]
is a disjoint union (indexed by the maximal Satake faces $\Phi_\Pi$ lying in each boundary component $\Pi$) and that balls in $D^{(2)}/\sim$ of a uniform radius are compact.
\begin{proof}[Proof of Theorem \ref{thm:1:2}, second half]
    \textbf{Step (1).} By the discussion following Lemma \ref{lem:5:2} and since there are finitely many Satake vertices in $D'$, we can choose $\delta>0$ and $r'>0$ such that
    \[
        D^{(2),0} = \bigsqcup_{\alpha}\left(B(\alpha,r')\cap \bigcup_{F_\alpha}N(F_\alpha,\delta)\right)
    \]
    is a disjoint union (over the Satake vertices $\alpha$), and remains so if $\delta$ is replaced by $2\delta$. Here $F_\alpha$ ranges over faces precisely incident with $\alpha$, and each component $B(\alpha,r')\cap \bigcup_{F}N(F,\delta)$ meets no faces not incident with its $\alpha$. By Proposition \ref{prop:4:3}, $\mathfrak{b}_\alpha$ is invariant under the Satake cycle of $\alpha$, and one shows exactly as in the classical hyperbolic case that balls of radius $\delta$ in in $D^{(2),0}/\sim$ are compact.
    
    \textbf{Step (2).} If $\alpha,\alpha'$ lie in the interior of the same boundary component $\Pi$, then
    \[
    C^{-1}\mathfrak{b}_{\alpha'}<\mathfrak{b}_{\alpha}<C\mathfrak{b}_{\alpha'},
    \]
    for some $C>1$. Using Lemma \ref{lem:5:6}, we choose $\epsilon$ and $r''>0$ so that
    \[
    D^{(2),1} = \bigsqcup_{\eta}\left(B(\alpha_\eta,r'')\cap \bigcup_{F_\eta} N(F_\eta,\epsilon)\backslash D^{(2),0}\right),
    \]
    is a disjoint union (over Satake edges $\eta$), remains so if $\epsilon$ is replaced by $2\epsilon$, and each component does not meet faces not incident with its $\eta$. Again, its image in the quotient has compact $\epsilon$-balls.

    \textbf{Step (3).} By the same comparability argument, there is \(r'''>0\) so that
    \[
        D^{(2),2} = \bigsqcup_{\Phi}\left(B(\alpha_\Phi,r''')\cap \bigcup_{F_\Phi} N(F_\Phi,\epsilon)\backslash \left(D^{(2),0}\cup D^{(2),1}\right)\right),
    \]
    is a disjoint union and each component meets only faces incident with its $\Phi$. For some \(\epsilon'>0\), balls of radius \(\epsilon'\) in $D^{(2),2}/\sim$ are compact.

    Finally, the compacrability argument allows us to choose $r>0$ small enough ensuring
    \[
    D^{(2)} =\bigcup_{\Pi} B(\alpha_{\Phi_\Pi},r)\subset D^{(2),0}\cup D^{(2),1}\cup D^{(2),2},
    \]
    so $D^{(2)}$ is a disjoint union with uniformly compact balls in the quotient, completing the proof.
\end{proof}
Combining the constructions of Subsections~\ref{subsec:5:1} and~\ref{subsec:5:2} yields the full proof of Theorem~\ref{thm:1:2}.
\vspace{12pt}
\section{An Example of a Dirichlet-Selberg Domain}\label{Sec:6}
In this section we exhibit an explicit finite-volume, complete $\mathcal X_3$-orbifold by gluing together a Dirichlet-Selberg domain along its facets.
\begin{exm}\label{exm:6:1}
    Let \(D\subset\mathcal X_3\) be the projective $5$-simplex whose six vertices lie on the Satake boundary and are given by the rank-one matrices
    \[
    \alpha_{1,2} = \left(\begin{array}{ccc}
        1 & \pm 1 & 0 \\
        \pm 1 & 1 & 0 \\
        0 & 0 & 0
    \end{array}\right),\ \alpha_{3,4} = \left(\begin{array}{ccc}
        1 & 0 & \pm 1 \\
        0 & 0 & 0 \\
        \pm 1 & 0 & 1
    \end{array}\right),\ \alpha_{5,6} = \left(\begin{array}{ccc}
        0 & 0 & 0 \\
        0 & 1 & \pm 1 \\
        0 & \pm 1 & 1
    \end{array}\right).
    \]
    Equivalently, under the identification of type-one component of $\partial_S\mathcal{X}_3$ with \(\mathbb{R}\mathbf{P}^2\), these correspond to \(\alpha_{1,2}=[1:\pm1:0]\), \(\alpha_{3,4}=[1:0:\pm1]\), \(\alpha_{5,6}=[0:1:\pm1]\). Label the unique facet of \(D\) missing \(\alpha_i\) by \(F_i\), for \(i=1,\dots,6\).

    Define three elements of \(SL(3,\mathbb{R})\),
    \[
    a = \left(\begin{array}{ccc}
        \frac{1}{2} & \frac{1}{2} & 0 \\
        \frac{1}{2} & -\frac{1}{2} & 1 \\
        \frac{1}{2} & -\frac{1}{2} & -1
    \end{array}\right),\ b = \left(\begin{array}{ccc}
        -\frac{1}{2} & 1 & \frac{1}{2} \\
        -\frac{1}{2} & -1 & \frac{1}{2} \\
        \frac{1}{2} & 0 & \frac{1}{2}
    \end{array}\right),\ c = \left(\begin{array}{ccc}
        -1 & \frac{1}{2} & -\frac{1}{2} \\
        0 & \frac{1}{2} & \frac{1}{2} \\
        1 & \frac{1}{2} & -\frac{1}{2}
    \end{array}\right),
    \]
    and set $\Gamma_0 = \{a,b,c,a^{-1},b^{-1},c^{-1}\}$. One checks that
    \[
        a.F_6 = F_1,\ b.F_2 = F_3,\ c.F_4 = F_5,
    \]
    and that each facet $F_i$ lies in the bisector $\mathrm{Bis}(I,g_i.I)$ for the corresponding generator $g_i\in \Gamma_0$. Hence $D$ is the Dirichlet-Selberg domain
    \[
    D = DS(I,\Gamma_0)\subset\mathcal{X}_3.
    \]

    The $15$ ridges $r_{ij} = F_i\cap F_j$ (for $1\leq i<j\leq 6$) break into five cycles under the action of $\Gamma_0$:
    \[
    \begin{split}
        & r_{56}\xrightarrow{a} r_{12}\xrightarrow{b} r_{34}\xrightarrow{c} r_{56}, \\
        & r_{14}\xrightarrow{a^{-1}} r_{36}\xrightarrow{b^{-1}} r_{25}\xrightarrow{c^{-1}} r_{14}, \\
        & r_{26}\xrightarrow{a} r_{16}\xrightarrow{a} r_{13}\xrightarrow{b^{-1}} r_{26}, \\
        & r_{24}\xrightarrow{b} r_{23}\xrightarrow{b} r_{35}\xrightarrow{c^{-1}} r_{24}, \\
        & r_{46}\xrightarrow{c} r_{45}\xrightarrow{c} r_{15}\xrightarrow{a^{-1}} r_{46}.
    \end{split}
    \]
    By direct computation of the invariant angle function \cite{Du2024-dp} one finds, for the first cycle, \(\theta_{inv}(r_{12})=\theta_{inv}(r_{34})=\theta_{inv}(r_{56}) = \frac{2\pi}{3}\), whence the total angle sum is \(2\pi\). For each of the remaining four cycles the sum of the (Riemannian) dihedral angles is \(\pi\). Thus \(D\) satisfies the angle-sum condition for Dirichlet-Selberg domains.

    Consequently, gluing the facets of \(D\) via the identifications in \(\Gamma_0\) produces a complete, finite-volume $\mathcal X_3$-orbifold \(M=D/\sim\).
\end{exm}
By Theorem \ref{thm:1:2}, the orbifold $M$ of Example \ref{exm:6:1} is complete. Hence, Poincar\'e's Fundamental Polyhedron Theorem yields the following presentation of its fundamental group.
\begin{cor}
    Let
    \[
    a = \left(\begin{array}{ccc}
        \frac{1}{2} & \frac{1}{2} & 0 \\
        \frac{1}{2} & -\frac{1}{2} & 1 \\
        \frac{1}{2} & -\frac{1}{2} & -1
    \end{array}\right),\ b = \left(\begin{array}{ccc}
        -\frac{1}{2} & 1 & \frac{1}{2} \\
        -\frac{1}{2} & -1 & \frac{1}{2} \\
        \frac{1}{2} & 0 & \frac{1}{2}
    \end{array}\right),
    \]
    and let $\Gamma = \langle a,b\rangle$. Then $\Gamma$ is a lattice in $SL(3,\mathbb{R})$ with presentation
    \[
    \Gamma = \langle a,b|(aba^{-1}b^{-1})^2,(ababa)^2,(a^2b^{-1})^2,(ab^3)^2\rangle.
    \]
\end{cor}
Since non-uniform lattices in \(SL(3,\mathbb{R})\) are quasi-isometric to \(SL(3,\mathbb{R})\) itself \cite{bridson2013metric}, the group \(\Gamma\) above is \textbf{not} Gromov-hyperbolic.

Next we describe the thin part of $M = \mathcal{X}_3/\Gamma$. Each vertex \(\alpha_i\) of \(D\) determines a one-dimensional subspace of $\mathbb{R}^3$, and these subspaces span $18$ full flags, corresponding to $1$-simplices or \textbf{Weyl chambers} in the visual boundary $\partial_\infty\mathcal{X}_3$. These flags break into three $\Gamma_0$-orbits and the associated minimal parabolic subgroups can be read off from the face-pairing data. For instance, considering the minimal parabolic subgroup $P = P_{V_\bullet}$, where
\[
V_\bullet = \mathrm{span}(\mathbf{e}_1+\mathbf{e}_3)\subset\mathrm{span}(\mathbf{e}_1+\mathbf{e}_2,\mathbf{e}_1+\mathbf{e}_3)\subset\mathbb{R}^3.
\]
Its unipotent radical \(P_0\) is torsion-free, satisfying \(P/P_0\cong K_4\). One finds generators
\[
u = \left(\begin{array}{ccc}
    1 & 1 & -1 \\
    0 & 0 & 1 \\
    0 & -1 & 2
\end{array}\right),\ v = \left(\begin{array}{ccc}
    0 & 0 & -1 \\
    -1 & 1 & -1 \\
    1 & 0 & 2
\end{array}\right),\ w = \left(\begin{array}{ccc}
    0 & 0 & -1 \\
    1 & 1 & 1 \\
    1 & 0 & 2
\end{array}\right),
\]
and the presentation
\[
P_0 = \langle u,v,w\mid [u,v]w^{-2},[u,w],[v,w]\rangle,
\]
so \(P_0\cong\pi_1(\mathbf{T}^2\rtimes_{\varphi}\mathbf{S}^1)\), the fundamental group of the mapping torus of $\varphi = \left(\begin{smallmatrix}
    1 & 2\\ 0 & 1
\end{smallmatrix}\right)$. Hence, each minimal-parabolic cuspidal end of $M$ is homeomorphic to
\[
\mathbb{R}_+^2\times((\mathbf{T}^2\rtimes_{\varphi}\mathbf{S}^1)/K_4).
\]
\begin{rmk}
    By Selberg's Lemma, \(M\) admits finite-degree manifold covers. For example, the kernel of the surjective homomorphism
    \[
    \Gamma_1 = \ker\left(\Gamma\hookrightarrow SL(3,\mathbb{Z}[\tfrac{1}{2}])\to SL(3,\mathbb{Z}/3\mathbb{Z})\right)
    \]
    is torsion-free and of index \(\lvert SL(3,\mathbb{Z}/3\mathbb{Z})\rvert=5616\). It remains an interesting question whether \(M\) admits a smaller-degree manifold cover.
\end{rmk}
\vspace{12pt}
\section{Future Directions}
Most of our constructions and results have been developed in the setting of the symmetric space \(\mathcal X_n\), but our proof of Theorem~\ref{thm:1:2} was carried out in detail only for \(\mathcal X_3\). We expect that the same arguments extend to arbitrary \(\mathcal X_n\) with help of the combinatoric structure of finite-volume Dirichlet-Selberg domains and properties of Busemann-Selberg functions.

A second, more ambitious direction is to remove the finite-volume condition.  Infinite-volume Dirichlet-Selberg domains arise from a much larger class of discrete subgroups of $SL(n,\mathbb{R})$, notably including various hyperbolic subgroups (e.g. surface group and knot group representations into $SL(n,\mathbb{R})$ \cite{long2011zariski,doi:10.1080/10586458.2011.565258}. Two new obstacles appear:
\begin{itemize}
  \item Infinite-volume polytopes admit infinitely many nonempty Satake boundary components (cf. Corollary~\ref{col:3:1}). By Proposition~\ref{prop:4:1}, a given horoball based at a type-one component meets infinitely many higher-type components, so one cannot trim away all ``higher-type intersections'' using only finitely many higher-type horoballs.
  
  \item As one approaches the Satake boundary inside an infinite-volume polyhedron, some Riemannian dihedral angles may tend to \(0\) or \(\pi\) (cf.\ Lemma~\ref{lem:3:2}). Therefore we lose the face-separation argument of Lemma~\ref{lem:5:6}, and would need new methods to prove the uniform compactness of balls, especially by exploiting the angle-sum condition.
\end{itemize}
Overcoming these issues - perhaps with additional techniques - could lead to a complete extension of our main theorem to the infinite-volume scenario.
\vspace{12pt}
\appendix
\section{An Inequality for Interlaced Sequence Deviations}
\begin{lem}\label{lem:append}
    Let \(n\) and \(k\) be positive integers with \(k<n\). Suppose 
    \[
    a_1 \ge a_2 \ge \cdots \ge a_n
    \quad\text{and}\quad
    b_1 \ge b_2 \ge \cdots \ge b_{n-k}
    \]
    are real numbers satisfying the interlacing condition,
    \[
    a_i\geq b_i\geq a_{i+k},\ i=1,\dots,n-k.
    \]
    Define the averages $\bar{a} = \frac{1}{n}\sum a_i$ and $\bar{b} = \frac{1}{n-k}\sum b_i$. Then,
    \[
    \sum_{i=1}^n (a_i - \bar{a})^2\geq \sum_{i=1}^{n-k}(b_i - \bar{b})^2.
    \]
\end{lem}
\begin{proof}
    We proceed by induction on \(k\). The base case \(k=1\) asserts that
    \[
    a_1\geq b_1\geq a_2\geq\dots\geq b_{n-1}\geq a_n.
    \]
    Fix such an interlacing \((a_i)\) and \((b_i)\). Since the squared deviation $(b_1,\dots, b_{n-1})\mapsto\sum_{i=1}^{n-1}(b_i - \bar{b})^2$ is convex, its restriction to $[a_2,a_1]\times\dots\times[a_{n-1},a_n]$ attains the maximum at a certain corner. Furthermore, the maximum is attained when the numbers $b_i$ are pairwise distinct. Hence up to reordering, \(\{b_1,\dots,b_{n-1}\}\) must equal \(\{a_1,\dots,a_n\}\setminus\{a_j\}\) for some \(1\le j\le n\). 
    
    A direct calculation then shows
    \[
    \sum_{i=1}^n\left(a_i - \bar{a}\right)^2 - \sum_{i\neq j}\left(a_i - \frac{\sum_{i\neq j}a_i}{n-1}\right)^2 = \frac{n-2}{n-1}(a_j - \bar{a})^2 \geq 0,
    \]
    which establishes the case \(k=1\).

    For general $k$, one has a refined sequence $\{b_1',\dots, b_{n-k+1}'\}$ satisfying $b_{i-1}\geq b_i'\geq b_i$ and $a_i\geq b_i'\geq a_{i+k-1}$. Applying the induction assumption to $\{a_i\}$ and $\{b_i'\}$ yields the desired inequality.
\end{proof}
\section{An Analytic Criterion for Finite Volume}
Definition \ref{defn:3:1} is indeed equivalent to the actual finite-volume condition:
\begin{prop}\label{lem:append:2}
    Let \(\mathbf D\subset\mathbf{P}(\mathrm{Sym}_n(\mathbb{R}))\) be a finitely-sided projective convex polytope, and \(D\subset\mathcal X_{n,\mathrm{proj}}\) be its restriction to $\mathcal{X}_n$. Then \(D\) has finite Riemannian volume in \(\mathcal X_n\) if and only if \(\mathbf D\subset\overline{\mathcal X_n}\).
\end{prop}
\begin{proof}
    We describe the Riemannian volume form on $\mathcal{X}_n$ by the standard projective volume form (see e.g. \cite{eberlein1996geometry}):
    \[
    d\mu(x_{ij}) = \frac{\iota_E\bigwedge_{i\leq j}dx_{ij}}{(\det(x_{ij}))^{(n+1)/2}},\ E = \sum_{i\leq j}x_{ij}\frac{\partial }{\partial x_{ij}}.
    \]
    \textbf{Necessity}. If \(\mathbf D\not\subset\overline{\mathcal X_n}\), then there is a boundary point \(X_0\in\partial_S D\) with \(\det X_0=0\) and a small projective-neighborhood \(U\ni X_0,\ U\subset\mathbf D\).  On \(U\cap \partial_S\mathcal{X}_n\) the denominator \(\det X\) vanishes to first order, so \(\int_{U\cap \overline{\mathcal{X}_n}} d\mu = \infty\). Hence \(D\) cannot have finite volume.
    
    \textbf{Sufficiency}. Conversely, assume \(\mathbf D\subset\overline{\mathcal X_n}\). We show each boundary neighborhood contributes a finite amount to the volume integral; compact interior patches are manifestly finite.

    Fix a boundary stratum of type \(n-k\).  After conjugating, we may take
    \[
    X_0 = \mathrm{diag}(I_k,O),\quad k<n.
    \]
    Introduce homogeneous coordinates near \(X_0\):
    \[
    X = X_0 + \begin{pmatrix}
      A & B\\
      B^{\mathsf T} & C
    \end{pmatrix},
    \quad
    A\in\mathrm{Sym}_k(\mathbb{R}),\ \mathrm{tr}(A) = 0,\  B\in M_{k,n-k}(\mathbb{R}),\ C\in\mathrm{Sym}_{n-k}(\mathbb{R}).
    \]
    Positivity of \(X\) forces \(C\) to lie in a convex polytope of the projective cone. Writing \(C = t\,Y\) with \(t\ge0\) and \(Y\in\mathcal X_{n-k}\), then $Y\in D_0\subset \mathcal{X}_{n-k}$. For fixed $Y\in D_0$, positive-definiteness along with the polyhedral property together show that entries in $B$ are $O(t)$, therefore $\det(X) = t^{n-k}\det(Y) + O(t^{n-k+1})$. By the compactness of $\overline{D_0}$ one has
    \[
    \det(X)\geq \frac{1}{2}t^{n-k}\det(Y),
    \]
    in a sufficiently small neighborhood.
    
    The volume form factorizes (up to a bounded Jacobian) as
    \[
    \iota_E\!\bigwedge d x_{ij}
    \;\approx\;
    \left(\iota \bigwedge dA\right)\left(\bigwedge dB\right)\left(\bigwedge dC\right).
    \]
    For fixed $C$, the positive-definiteness show that
    \[
    \int \bigwedge dB\leq 2^{k(n-k)}(\det(C))^{k/2} = 2^{k(n-k)}(\det(Y))^{k/2}t^{k(n-k)/2}.
    \]
    For fixed $t$, $B = O(t)$, $\int \bigwedge dB = O(t^{k(n-k)})$, thus one estimates
    \[
    \int \bigwedge dB\leq K (\det(Y))^{k/2}t^{k(n-k)},\ K<\infty.
    \]
    Meanwhile, $C = tY$ contributes
    \[
    \bigwedge dC = t^{(n-k-1)(n-k+2)/2}dt\wedge \left(\iota_E \bigwedge_{j\geq i\geq k+1} dy_{ij}\right),\ E = \sum_{i\leq j}y_{ij}\frac{\partial}{\partial y_{ij}}.
    \]
    Hence the singularity near $X$ is integrable:
    \begin{align*}
        & \int_{U_\epsilon}d\mu = \int_{U_\epsilon}\frac{\iota_E \bigwedge dx_{ij}}{\det(X)^{(n+1)/2}} \leq \int_{U_\epsilon}\frac{2^{(n+1)/2}\iota_E \bigwedge dx_{ij}}{\det(tY)^{(n+1)/2}} \\
        & = \int_{|x_{ij}|<\epsilon}\iota\bigwedge dA\int \bigwedge dB\int_0^\epsilon dt\int_{D_0}\frac{2^{(n+1)/2}t^{(n-k-1)(n-k+2)/2}\iota_E\bigwedge_{i\leq j}dy_{ij}}{\det(tY)^{(n+1)/2}} \\
        & \leq (2\epsilon)^{(k-1)(k+2)/2}\int_0^\epsilon dt\int_{D_0}K(\det(Y))^{k/2}t^{k(n-k)}\frac{2^{(n+1)/2}t^{(n-k-1)(n-k+2)/2}\iota_E\bigwedge_{i\leq j}dy_{ij}}{\det(tY)^{(n+1)/2}} \\
        & = K'\epsilon^{(k-1)(k+2)/2}\int_0^\epsilon dt\int_{D_0}\frac{t^{k(n-k)/2-1}\iota_E\bigwedge_{i\leq j}dy_{ij}}{\det(Y)^{(n-k+1)/2}} \\
        & = \frac{K'\epsilon^{k(n+1)-1}}{k(n-k)}\mathrm{Vol}(D_0)<\infty,
    \end{align*}
    using the induction assumption for $\mathcal{X}_{n-k}$. Covering \(\mathbf D\) by finitely many such boundary charts plus an interior compact set shows \(\int_D d\mu<\infty\).
\end{proof}
% k(n-k)/2 + k(k+1)/2 - 1
\bibliography{reference}{}
\bibliographystyle{alpha}
\end{document}